\documentclass[letterpaper,10pt]{amsart}

\usepackage[all]{xy}                    

\CompileMatrices                            

\UseTips                                    

\input xypic
\usepackage[bookmarks=true]{hyperref}       

\usepackage{amssymb,latexsym,amsmath,amscd}
\usepackage{xspace}
\usepackage{color}
\usepackage{graphicx}
\usepackage{dsfont}
\usepackage{enumitem}
\usepackage{xcolor}
\usepackage{float}
\usepackage[T1]{fontenc}
\usepackage[utf8]{inputenc}

\reversemarginpar

\DeclareMathOperator{\Res}{Res}

\theoremstyle{plain}
\newtheorem{theorem}{Theorem}[section]
\newtheorem*{theorem*}{Theorem}
\newtheorem{proposition}[theorem]{Proposition}

\newtheorem{lemma}[theorem]{Lemma}

\theoremstyle{definition}
\newtheorem{definition}[theorem]{Definition}

\newtheorem{notation}[theorem]{Notation}

\newtheorem{remark}[theorem]{Remark}
\newtheorem{example}[theorem]{Example}

\newcommand{\enm}[1]{\ensuremath{#1}}          %
\newcommand{\op}[1]{\operatorname{#1}}
\newcommand{\cal}[1]{\mathcal{#1}}

\newcommand{\EE}{\enm{\mathbb{E}}}

\newcommand{\NN}{\enm{\mathbb{N}}}

\newcommand{\PP}{\enm{\mathbb{P}}}

\newcommand{\TT}{\enm{\mathbb{T}}}
\newcommand{\KK}{\enm{\mathbb{K}}}

\newcommand{\Ee}{\enm{\cal{E}}}

\newcommand{\Ii}{\enm{\cal{I}}}

\newcommand{\Ll}{\enm{\cal{L}}}

\newcommand{\Oo}{\enm{\cal{O}}}

\newcommand{\Ss}{\enm{\cal{S}}}

\newcommand{\Uu}{\enm{\cal{U}}}

\renewcommand{\phi}{\varphi}
\renewcommand{\theta}{\vartheta}
\renewcommand{\epsilon}{\varepsilon}

\newcommand{\Aut}{\op{Aut}}

\newcommand{\codim}{\op{codim}}

\newcommand{\Union}{\bigcup}

\renewcommand{\to}[1][]{\xrightarrow{\ #1\ }}

\newcommand{\old}[1]{}

\newtheorem{claim}{Claim}[theorem]

\title{Terracini Locus for three points on a Segre variety}
\date{}
\author{Edoardo Ballico} \address{Dipartimento di Matematica, Univ. Trento, Italy, edoardo.ballico@unitn.it}
\author{Alessandra Bernardi} \address{Dipartimento di Matematica, Univ. Trento, Italy, alessandra.bernardi@unitn.it}
\author{Pierpaola Santarsiero} \address{Dipartimento di Matematica, Univ. Trento, Italy, p.santarsiero-1@unitn.it}

\makeatletter
\@namedef{subjclassname@2020}{%
  \textup{2020} Mathematics Subject Classification}
\makeatother

\begin{document}

\thanks{The authors were partially supported by GNSAGA of INDAM}

\maketitle

\begin{abstract}
We introduce the notion of \emph{$r$-th Terracini Locus} of a variety and we compute it for at most three points on a Segre variety.
\end{abstract}

\section*{Introduction}

The celebrated Terracini Lemma \cite{ter, adl} is a well known and extremely powerful result in Algebraic Geometry that allows to compute the dimensions of $r$-th secant varieties of a given variety $X$ in terms of the dimensions of the sum of tangent spaces at $r$ generic points of $X$. If $X$ is the embedding of a variety $Y$  into a projective space via a complete linear system $\mathcal{L}$, then the  co-dimension of the $r$-th secant variety of $X$ is equivalent to the $h^0(Y,I_Z\otimes \mathcal{L})$ where $Z$ is a 0-dimensional scheme of $r$ double generic fat points (cf. e.g.
 \cite{gim, CGG2, devar}).
The classical apolarity theory \cite{ger,IK} is the very well known example in which the variety $X$ is  $Y=\mathbb{P}^n$ embedded via $\mathcal{O}(d)$ for which Alexander-Hirschowitz completely classified dimensions of all secant varieties to any Veronese variety (cf. \cite{ah2}). The only other complete classification is for secant varieties of Segre-Veronese embedding of products of $Y=(\mathbb{P}^1)$'s via $\mathcal{O}(d_1, \ldots, d_k)$ due to Laface-Postinghel (cf. \cite{LP}).  There is a vast literature in this field
(see e.g. \cite{CGG, gal, BBG, grv, Abo, AboB, ballico, BBC, Abrescia, AboB2, AMR} and references therein)  but almost anything has been said for the case in which the $0$-dimensional scheme of double fat points is not necessarily supported on generic points. Clearly if the points are not generic, the equivalence between  $h^0(Y,I_Z\otimes \mathcal{L})$  and the co-dimension of the secant variety of $X$ is not valid anymore. Indeed Terracini lemma states that the tangent space of a $r$-th secant variety of a variety $X$ at a generic point $Q\in \langle P_1, \ldots , P_r \rangle$, with $P_i\in X$ generic, is equal to the span of the tangent spaces of $X$ at $P_i$'s, but if the $P_i$'s are not generic one can only say that $\langle T_{P_1}X , \ldots , T_{P_r}X\rangle \subseteq T_Q(\sigma_r(X))$ where $T_Q(\sigma_r(X))$ is the tangent space at $Q\in\langle P_1, \ldots , P_r \rangle$ of the $r$-th secant variety of X. This phenomenon is related to the fact that $\codim\langle T_{P_1}X , \ldots , T_{P_r}X\rangle= h^0(Y,I_Z\otimes \mathcal{L})$ which may be higher than the one for  generic points.
If one considers the exact sequence
$$     0\to \Ii _Z \to \Oo_Y\to
\Oo_Z\to 0
$$
and then one tensorizes it by $\mathcal{L}$:
$$0\to \Ii _Z \otimes \mathcal{L} \to \mathcal{L}\to
\mathcal{L}_Z\to 0$$
the corresponding cohomology exact sequence 
$$0 \to 
H^0(Y, \mathcal{I}_Z\otimes \mathcal{L}) \to
H^0(Y,\mathcal{L})\to
H^0(Z, \mathcal{L}|_Z)\to H^1(Y, \mathcal{I}_Z\otimes \mathcal{L})\to 0$$
shows that 
$$
h^0(Y,\mathcal{L})-h^0(Y, \mathcal{I}_Z\otimes \mathcal{L}) = 
h^0(Z, \mathcal{L}|_Z)- h^1(Y, \mathcal{I}_Z\otimes \mathcal{L})$$ 
that is to say that the dimension of the span  of embedding of $Z$ via $\mathcal{L}$ can be computed as $$h^0(\mathcal{O}_Z)-h^1(\mathcal{I}_Z\otimes \mathcal{L})-1.$$

From this one easily see the role played by the $h^1(\Ii _Z\otimes \mathcal{L})$ in controlling the dependence of $ h^0(Y, Z \otimes \mathcal{L}) $.

\medskip

In this paper we fix our attention on the case of Segre varieties, i.e. the embedding of $Y=\mathbb{P}^{n_1}\times    \cdots   \times \mathbb{P}^{n_k}$ via $\mathcal{O}(1,\ldots ,1)$ and $Z$ a scheme of either 2 or 3 double fat points. We will define the key object that we will call the $r$-th Terracini Locus  that will essentially  contain all the subsets of $r$ points for which both $h^0(\Ii_Z(1, \ldots , 1))>0$ and $h^1(\Ii_Z(1, \ldots , 1))>0$.

We like to point out the geometric importance of the  $r$-th Terracini Locus.
Consider the so called $r$-th Abstract Secant variety of a Segre Variety $X\subset \mathbb{P}^N$:
$$Abs_r(X):=\{(Q; P_1, \ldots , P_r)\in \mathbb{P}^N \times X^r \, | \, Q\in \langle P_1, \ldots, P_r \rangle\}.$$
If one considers the first projection on $\mathbb{P}^N$ one gets that $Abs_r(X)$ projects onto what we can call an ``\,open part\,'' of the $r$-th secant variety of $X$, namely $\sigma_r^0(X):= \{Q \in \mathbb{P}^N\, | \, Q=\sum_{i=1}^r P_i, \, P_i \in X\}$:
$$T_r:Abs_r(X) \to \sigma_r^0(X).$$
We call such a projection $T_r$ the $r$-th \emph{Terracini map}. The differential of the $r$-th Terracini map is defined on each  point of $Abs_r(X)$ and the $r$-th Terracini locus is nothing else than a measure of the degeneracy of such a linear map.
Remark that any point of $Abs_r(X)$ is smooth since $X$ is smooth.

\medskip

In the first section of this paper we introduce the notation and we show that the second Terracini locus is empty, meaning that the second Terracini map is never degenerate.
The second section is a technical one where we concentrate all the core lemmas that will be needed in the sequel in order to prove our main theorem (Theorem \ref{i33}) that is a complete description of the $3$-rd Terracini locus. Section \ref{section:examples} is a crucial section where we show all the examples that will turn out the only cases in which the $3$-rd Terracini locus will be not empty. Section \ref{section:main} is devoted to the proof of the main theorem that essentially will be a discussion on why the already higlighted examples in Sections \ref{ex1} and \ref{a4.0} are the only non empty $3$-rd Terracini locus.

Let $Z\subset Y $ be a scheme of $ r\geq 2$ double points embedded via Segre in a $n $-dimensional multiprojective space $Y $.
In the last section we compute the maximal value $\max_{n>0,r\geq 2}\{ h^1\left(\Ii_{Z}(1,\dots,1)\right)>0\}.$
We will show that 
$$h^1\left(\Ii_Z(1,\dots,1)\right) \leq (r-1)(n+1)$$
and that equality holds if and only if $ Y=\PP^n$. Moreover, since $h^0\left(\Ii_{\PP^n}(1)\right)=0 $, we compute the maximal value of such dimension providing that also $h^0\left(\Ii_Z(1,\dots,1)\right)>0 $. 
Finally we prove that for any multiprojective space $Y $ of dimension  $n\geq 3 $, one can always find $r\geq 3 $ points $S\subset Y$ belonging to the corresponding $ r$-th Terracini locus. This conclusion might open to further investigation of the introduced locus.

\section{Notation}
We work over an algebraically closed field $\mathbb{K}$ of characteristic $ 0$.
In the following we will always deal with a multiprojective space of $ k>0$ factors  $Y$ of the form $$
Y:=\mathbb{P}^{n_1}\times \cdots \times \PP^{n_k}.$$
\begin{notation} 
Let $V_1,\dots,V_k $ be $\KK$-vector spaces of dimensions $n_1+1,\dots,n_k+1 $ respectively. Denote by $\nu $ the Segre embedding of $ Y$, which is defined as
\begin{align*}
\nu \colon &\PP(V_1)\times \cdots \times \PP(V_k)\rightarrow \PP(V_1\otimes \cdots \otimes V_k)\\
& ([v_1],\dots,[v_k])\mapsto [v_1\otimes \cdots \otimes v_k].
\end{align*}

\end{notation}
We will denote the Segre variety of $Y$ by $$X:=\nu(Y).$$
\begin{notation}\label{Notation:Yi}
We denote the projection of $Y $ onto the $i $-th factor by
$$\pi_i\colon Y\to \PP^{n_i}.$$
Fix $Y_i:=\mathbb{P}^{n_1}\times \cdots \widehat{\PP^{n_i}}\times \cdots \times \PP^{n_k} $, for some $1\leq i\leq k $. By $\eta_i $ we denote the map that projects $ Y$ onto $ Y_i$ forgetting the $ i$-th factor, i.e.
$$\eta_i\colon Y\to Y_i.$$  

The Segre embedding of $Y_i $ is denoted via $$\nu_i\colon Y_i\to \mathbb{P}(V_1 \otimes  \cdots \otimes  \hat V_i \otimes \cdots \otimes V_k).$$
\end{notation}

Let $ X\subset \mathbb{P}^N$ be an irreducible non-degenerate projective variety of dimension $ n$. Fix $r\geq 2 $. The $r $-th \emph{secant variety} $\sigma_r(X) $ of $ X$ is the Zariski closure of all $(r-1) $-planes spanned by $r $ points of $ X$. Namely 
$$ \sigma_r(X):=\overline{\Union_{p_1,\dots,p_r \in X}\langle p_1,\dots,p_r \rangle}.$$
Remark that $\dim\sigma_r(X)\leq \min\{r(n+1)-1,N \} $. If the previous inequality is strict the $ r$-th secant variety of $ X$ is said to be \emph{defective} with \emph{defect} $\min\{r(n+1)-1,N \} -\dim\sigma_r(X) $.

\begin{notation}\label{doble:points}
For any $p\in Y$,  denote by $(2p,Y)$ the first infinitesimal
neighborhood of $p$ in $Y$, i.e. the closed subscheme of $Y$ with $(\Ii _{p,Y})^2$ as its ideal sheaf. For any finite set
$S\subset Y$ let $(2S,Y):= \cup _{p,S} (2p,Y)$. We often write $2p$ and $2S$ instead of $(2p,Y)$ and $(2S,Y)$ if the dependence from $Y$ is clear.
\end{notation}

Remark that if $W\subseteq Y$ is a multiprojective subspace and $p\in W$, then 
$(2p,W) \neq (2p,Y)$ as schemes. In fact $\deg (2p,W) =\dim W+1$ and $\deg (2p,Y)
=\dim Y+1$. However $(2p,W) = (2p,Y)\cap W$ (scheme-theoretic intersection) and hence $(2p,W)\subseteq (2p,Y)$.
Thus for any finite set $S\subset W$ one has $(2S,W) = (2S,Y)\cap W$.

\begin{notation} For $k>0 $ fix the following notation:
\begin{itemize}
    \item $\epsilon_i:=(0,\dots,0,1,0,\dots,0)\in \NN^k $ is the $k$-uple given by all $ 0$'s but $ 1$ in the $ i$-th position;
    
    \item $ \hat{\epsilon}_i:=(1,\dots,1,0,1,\dots,1)\in \NN^k$ is the $k$-uple given by all $ 1$'s but $ 0$ in the $ i$-th position;
    
    \item $\epsilon_I$ is the $k$-uple having $1$'s in the places indexed by the finite set $I \subset \{1,\dots,k \} $ and $0$'s everywhere else;
    
    \item $\hat \epsilon_I$  is the $k$-uple having $0$'s in the places indexed by  $I$ and $1$'s everywhere else.
\end{itemize}
\end{notation}

Let $ Z\subset Y$ be a $0$-dimensional scheme. Fix a finite set $ I\subset\{1,\dots,k \}$ and $ H\in \vert\Oo(
\epsilon_I) \vert$. 
Denote by $ \Res_{H}(Z)$ the residue of $ Z$ with respect to $ H$, i.e. the $ 0$-dimensional scheme defined by the  ideal sheaf $ \Ii_Z \colon \Ii_H$. By $ H\cap Z$ denote the scheme-theoretic intersection of $ Z$ and $ H$.
The residual exact sequence of $ Z$ with respect to $ H$ is   
\begin{equation*}
    0\to \mathcal{I}_{\Res_{H}(Z)}\big(
    \hat \epsilon_I
    \big)\to \mathcal{I}_{Z}(1,\dots,1) \to \mathcal{I}_{H\cap Z,H}(1,\dots,1)\to 0.
\end{equation*}

Since throughout the paper we will deal with double points, we recall how to adapt the residual exact sequence in this case.

Take $Z:=(2S,Y) $, where $S\subset Y $ is a finite set, consider a proper subset $S'\subset S $ such that $\# (S'\cap H)>0 $ and denote by $S'':=S\setminus S' $. In this case, the scheme-theoretic intersection of $Z $ and $H $ is  $(2S',H) $, while $\Res_H(Z) $ is the zero-dimensional scheme of what is left once we specialized $(2S',Y) $ into $H $, namely $$\Res_H(Z)=(S',Y)\cup (2S'',Y).$$

Since we will use the restricted exact sequence only with $ 0$-dimensional schemes, we recall the restriction sequence of $ Z$ with respect to $ Y$, namely
 \begin{equation*}
     0\to \Ii _Z(1,\dots,1) \to \Oo_Y(1,\dots,1)\to
\Oo_Z(1,\dots,1)\to 0.
 \end{equation*}
 Since this exact sequence is defined for any embedding of $ Y$, one can also use it for any line bundle given by $ \Oo(
 \epsilon_I
 )$ instead of the one given by $ \Oo(1,\dots,1)$, i.e.
  \begin{equation*}
     0\to \Ii _Z\big(
 \epsilon_I
 \big) \to \Oo_Y\big(
  \epsilon_I
 \big)\to
\Oo_Z\big(  
 \epsilon_I
\big)\to 0.
 \end{equation*}

The corresponding cohomology exact sequence 
$$0 \to H^0(Y, \mathcal{I}_Z(\epsilon_I)) \to H^0(Y,\mathcal{O}_Y(\epsilon_I))\to H^0(Y, \mathcal{O}_Z)\to H^1(Y, \mathcal{I}_Z(\epsilon_I))\to 0$$
shows that the dimension of the subspace
$$\langle \nu(Z)\rangle= \mathbb{P}(H^0(Y, \mathcal{I}_Z(\epsilon_I)^\perp))$$ 
is equal to $h^0(\mathcal{O}_Z(\epsilon_I))-h^1(\mathcal{I}_Z(\epsilon_I))-1$.

From this one easily see the role played by the $h^1(\Ii _Z(\epsilon_I))$ in controlling the dependence of the multilinear forms passing through $Z$ and therefore the speciality of $Z$.

\begin{notation}\label{virtual:defect}
For any zero-dimensional scheme $Z\subset Y$ set $$\delta (Z,Y):= h^1(\Ii _Z(1,\dots ,1)).$$ 
If $W\subseteq Y$ is a multiprojective subspace 
set $\delta (Z,W):= h^1(W,\Ii _{Z,W}(1,\dots ,1)).$ We remark that $\delta (Z,W)=\delta (Z,Y)$.
Sometimes we will write $\delta (Z)$ instead of $\delta(Z,Y)$, when the dependence from $ Y$ is clear. 

In particular for any finite set $S\subset W$ there are defined the integers $\delta ((2S,Y),Y)$, $\delta ((2S,W),W)$ and $\delta ((2S,W),Y)$.
Set $$\delta (2S,Y):= \delta ((2S,Y),Y) \hbox{ and }\delta (2S,W):= \delta ((2S,W),W).$$ Clearly $\delta (2S,W) = \delta ((2S,W),Y)$. For the specific case of double fat points, $\delta(2S,Y)$ will be called the Terracini defect of $S$ in $Y$ (see Definition \ref{terraccini:loci}).
\end{notation}

Remark that by using the above exact sequence with $H\in |\Oo_Y(\epsilon _i)|$, $i\in \{1,\dots ,k\}$, for a finite set $S\subset H$ such that $h^0\left(H,\Ii_{S\cap H,H}(1,\dots
,1)\right)=0$
one has $\delta (2S,Y) = \delta (2S,H) +h^0(\Ii_S(\hat{\epsilon}_i))$.

\begin{notation} Let $ Y$ be any multiprojective space. For all positive integers $r$, denote by $S(Y,r)$ the set of all subsets of $Y$ with
cardinality $r$.
\end{notation}
Let 
$ S\in S(Y,r)$ 
be a set of $r>0$ distinct points.
The \emph{minimal multiprojective space containing $ S$} is $ Y':=\PP^{n'_1}\times \cdots \times \PP^{n'_{k'}}\subseteq Y$ 
where $\PP^{n'_i}:=\langle \pi_i(S) \rangle $, $ i=1,\dots,k'$ and  $ k'\leq k$.
The integer $ k'\leq k$ is the maximum integer such that $ \#\pi_i(S)>1$ for all $ i$'s.
Clearly $ \PP^{n'_1}\times \cdots \times \PP^{n'_{k'}}\times \{ o_{k'+1}\}\times \cdots \times \{o_{k} \}\cong \PP^{n'_1}\times \cdots \times \PP^{n'_{k'}}$ and in general $ Y'$ is such that $k'\leq k  $ and $n'_i>0$ for all $i$'s.

\smallskip

We have now introduced all the necessary tools to define the \textbf{$r$-th Terracini locus} that will be the main actor of the present paper.

\begin{definition}\label{terraccini:loci}  For all positive integers $r$ 
and for any multiprojective space $Y$, define
$$\TT _1(Y,r):= \left\{S\in S(Y,r)\mid h^0\left(\Ii _{(2S,Y)}(1,\dots ,1)\right) >0\ \mathrm{and} \ \delta(2S,Y)>0\right\}.$$
We will call the \emph{$r$-th Terracini locus} $\TT(Y,r) $ of all $r$-uple of points of $Y$ the set

\begin{table}[H]
    \centering
    \begin{tabular}{|c|}
    \hline \\
        $\TT(Y,r):=\left\{S\in \TT_1(Y,r) \, | \, Y \hbox{ is the minimal multiprojective space containing } S\right\}.$  \\
        \\
        \hline
    \end{tabular}
\end{table}
 
For any $S\in \TT_1(Y,r)$ we call the integer $\delta (2S,Y)$  introduced in Notation \ref{virtual:defect} the \emph{$r$-th Terracini defect of $S$ in $Y$} or the \emph{defect
of $2S$ in $Y$}. 

If $Y'\subseteq Y$ is the minimal multiprojective space containing a set $S$, the integer $\delta
(2S,Y'):=h^1(Y',\Ii_{(2S,Y')}(1,\dots ,1))$ is called the \emph{absolute $r$-th Terracini defect of $S$}.
\end{definition}

\subsection{The $2$-nd Terracini locus is empty}
\phantom{a}

In this subsection we  prove that no sets of two distinct points $ S\subset Y$ such that $ Y$ is the minimal multiprojective space containing $ S$, is contained in the $2$-nd Terracini locus $ \TT(Y,2)$.

\begin{proposition}\label{d1} The $2$-nd Terracini locus $\TT(Y,2)$ 
is empty for any multiprojective space $Y$.
\end{proposition}

\begin{proof}
Let $ S\in S(Y,2)$ be such that $ Y$ is the minimal multiprojective space containing $ S$. So $\#\pi_i(S)=2 $ for all $i $'s and $Y\cong (\PP^1)^k$ for some $k\geq 1 $. By definition of $\TT(Y,2)$, we need to prove that either $h^0\left(\Ii _{(2S,Y)}(1,\dots ,1)\right) =0$ or $h^1\left(\Ii _{(2S,Y)}(1,\dots
,1)\right) =0$.
Clearly if $k=1 $ then $h^0\left(\Ii _{(2S,\mathbb{P}^1)}(1)\right) =0$. If $k=2 $, then $h^0(\Ii_{(2S,Y)}(1,1))=0 $ since $ S$ can be seen as a general subset of $ 2$ distinct points by the action of $(\mathrm{Aut}(\PP^1))^2 $ and a general $ 2\times 2$ matrix has rank $ 2$.

Let $k\ge 3$. Let $E$ be the set of all $A\subset Y$ such that $\#A=\#\pi _i(A)=2$ for all $i$'s.
The group $(\mathrm{Aut}(\PP^1))^k$ acts transitively on $E$. Thus $S$ may be considered as a general subset of $Y$ with
cardinality $2$. Since $\dim \sigma _2(X) =2k+1$ for all $k\ge 3$ (cf. \cite{laf,v}), Terracini's lemma gives $h^1\left(\Ii
_{(2S,Y)}(1,\dots ,1)\right)=0$.
\end{proof}

Let's point out some consequence of Proposition \ref{d1} in terms of identifiability of rank $ 2$ tensors. Since we are dealing with finite subsets $ S$ of two distinct points, the minimal multiprojective space containing $ S$ is $Y=(\PP^1)^k $ for some $ k\geq1$, which is equivalent to say that $\#\pi_i(S)=2 $ for all $i $'s. Thus we may look at $ S:=\{p_1,p_2\}$ as a general set of two distinct points thanks to the action of $ (\mathrm{Aut}(\PP^1))^k$.\\
The emptiness of the $2 $-nd Terracini locus $ \TT(Y,2)$ corresponds to a non degeneracy of the second Terracini map $ T_2:Abs_r(X) \to \sigma_r^0(X)$ for any $X=\nu((\PP^1)^k) $, with arbitrary $k\geq 2$.
Since we are working with general points, the condition $$h^0(\Ii_{(2S,Y)}(1,\dots,1)) >0$$ corresponds to prescribe that the $ 2$-nd secant variety $ \sigma_2(X)$ does not fill the ambient space. This condition together with $$ h^1(\Ii_{(2S,Y)}(1,\dots,1))>0$$ are equivalent to ask that the dimension of the tangent space $T_q\sigma_2(X) $ at a general $ q\in \PP^{2^k-1}$ such that $q\in \langle \nu(p_1),\nu(p_2) \rangle $, is strictly less than $2(k+1)-1$. 

We are therefore looking for general rank $ 2$ tensors such that the corresponding secant variety does not fill the ambient space and such that the generic fiber of $T_2 $ is positive dimensional. 

The emptiness of $ \TT(Y,2)$ establishes that
given $\sigma_2(X)$ that does not fill the ambient space, all points of $ \sigma^0_2(X)\setminus X$ are identifiable.

\section{Main lemmas}
\begin{remark}\label{m1} 

If $A\subset B\subset Y$ are zero-dimensional schemes, then
\begin{equation}\label{eqm1}
\delta (A,Y)\le \delta (B,Y)\le \delta(A,Y)+\deg(B) -\deg(A).
\end{equation}
Indeed the first inequality is clear since $ A\subset B$. Moreover we remark that if $ A\subset B$ then $ h^0(\Ii_B(1,\dots,1))\leq h^0(\Ii_A(1,\dots,1))$. So by the restriction exact sequences of both $ A$ and $ B$ with respect to $ Y$, we get the second inequality. In particular for all $S'\subset S\subset Y$ we have
\begin{equation}\label{eqm2}
\delta (2S',Y)\le \delta (2S,Y)\le \delta(2S',Y)+(\#S-\#S')(\dim Y+1).
\end{equation}
\end{remark}

\begin{definition}
We say that a finite set $S\subset Y$ is 
\emph{minimally Terracini} if $\delta (2S,Y)>0$ and $\delta (2S',Y) =0$ for all $S'\subsetneq S$.
\end{definition}

The following key lemma will be used to prove a sort of concision for the Terracini locus of a finite set $S\subset Y$.
By Remark \ref{m1}, if $S'\subset S$ is a scheme of $ r$ double points, the $ r$-th Terracini defect $ \delta(2S',Y)$ is smaller than $\delta(2S,Y) $.\\ In the following lemma we fix the finite set $ S\in S(W,r)$ and we compare the behaviour of the two $ r$-th Terracini defects $\delta(2S,W) $ and $ \delta(2S,Y)$, where $ W\subsetneq Y$ is a smaller multiprojective space.
Since $Y $ is no longer the minimal multiprojective space containing $ S\subset W\subsetneq Y$, the $ r$-th Terracini defect $\delta(2S,Y) $ may be bigger than $\delta(2S,W) $. In case \ref{a1disuguaglianze} of Lemma \ref{a1} we give an upper bound for $\delta(2S,Y) $ via $\delta(2S,W) $.  Case \ref{a1unfattore} can be considered as a strong version of concision because the achievement of equality $\delta(2S,W)=\delta(2S,Y) $ is telling that the defect of $ 2S$ is independent from the number of factors of the multiprojective space where $S$ is embedded.

\begin{lemma}\label{a1}
Let $W\subsetneq Y$ be multiprojective spaces. Let $S\subset W$ be a finite set. Then:
\begin{enumerate}[label=(\alph*)]
\item \label{a1disuguaglianze}
$ \delta(2S,W)\leq \delta(2S,Y)\leq \delta(2S,W)+(\#S-1)(\dim Y
-\dim W).$

\item\label{a1unfattore} If $W$ is isomorphic to a factor of $Y$ and $\nu (S)$ is linearly independent, then 
$ \delta(2S,W)=\delta(2S,Y)$.
\end{enumerate}
\end{lemma}

\begin{proof}
Since the restriction map $H^0(Y,\Oo _Y(1,\dots ,1)) \to H^0(W,\Oo_W(1,\dots ,1))$ is surjective and $(2S,W)\subseteq (2S,Y) $, the first inequality of part \ref{a1disuguaglianze} is the first inequality of \eqref{eqm1}. So we just need to prove the second inequality of \ref{a1disuguaglianze}: we will do it by induction on the integer $ \dim Y - \dim W$.

First assume $\dim Y=\dim W+1$. Thus there is $i\in \{1,\dots ,k\}$ such that $W\in |\Oo_Y(\epsilon _i)|$. Note that $W\cap
(2S,Y) = (2S,W)$ and that $\Res_W(2S,Y) = S$. Thus the residual exact sequence of $W$ gives the following
exact sequence
\begin{equation}\label{eqa1}
0 \to \Ii _S(\hat{\epsilon}_i) \to \Ii _{(2S,Y)}(1,\dots ,1) \to \Ii _{(2S,W)}(1,\dots ,1)\to 0.
\end{equation}
Since the restriction map $H^0(Y,\Oo _Y(1,\dots ,1)) \to H^0(W,\Oo_W(1,\dots ,1))$ is surjective,  $h^1\left(Y,\Ii
_{(2S,W)}(1,\dots ,1)\right) =h^1\left(W,\Ii
_{(2S,W)}(1,\dots ,1)\right)$. Since $S$ is a finite set, $h^i(\Ll ) =0$ for all $i>0$ and all line bundles $\Ll$ on $S$. The
long cohomology exact sequence of the exact sequence $$0\to \Ii _S(\hat{\epsilon}_i) \to \Oo_Y(\hat{\epsilon} _i)\to
\Oo_S(\hat{\epsilon} _i)\to 0$$ gives $h^2\left(\Ii _S(\hat{\epsilon} _i)\right) =h^2\left(\Oo_Y(\hat{\epsilon}_i)\right)=0$. Since $h^1\left(\Oo _Y(\hat{\epsilon}_i)\right) =0$ and $\Oo _Y\left(\hat{\epsilon}_i\right)$ is globally generated, $h^1\left(\Ii _S(\hat{\epsilon}_i))\right)\le \#S -1$. Thus \eqref{eqa1}
gives part \ref{a1disuguaglianze}.
Note that we have $h^1\left(W,\Ii _{(2S,W)}(1,\dots ,1)\right)
=  
h^1\left(Y,\Ii _{(2S,Y)}(1,\dots ,1)\right)$ if $h^1\left(\Ii _S(\hat{\epsilon}_i)\right)=0$. 

Now assume $\dim Y \ge \dim W +2$. We can always find a multiprojective space $ M$ such that $W\subsetneq M\subseteq Y$ and in particular we take $ M \in \vert \mathcal{O}_Y(\varepsilon_i) \vert$ for some $i $.
The inductive step follows by applying the codimension one case to the inclusion $M\subset Y $ and we conclude by applying the inductive assumption on the inclusion $ W\subset M$.

Assume that $ W$ is isomorphic to a factor of $ Y$, say $ Y\cong W\times Y'$. We will show \ref{a1unfattore} by induction on the number of factors of $ Y'$. Assume $ Y$ has two factors i.e. $Y=W\times \mathbb{P}^m $, for some $m>0$, where $W\cong W'\times \{o \} $ for some $ o\in \mathbb{P}^m$ and some positive dimensional projective space $W' $. We will work by induction on $m\geq1 $.\\
First assume $m=1 $, so $W \in \vert \mathcal{O}_Y(\varepsilon_2)\vert $ and in particular $W=\pi_2^{-1}(o) $ where $o \in \mathbb{P}^1 $.
We remark that the Segre embedding $\nu _2$ of $W$ can be seen as the restriction to $W$ of the Segre embedding of $Y$. Thus $\nu (S)$ is linearly independent if and only if $\nu _2(S)$ is linearly independent. Note that the linear independence of $\nu _2(S)$ is equivalent to $h^1\left(\mathcal{I}_S(1,0)\right)=0$ because $\pi _2(S)=\{o \}$. Since we already proved part \ref{a1disuguaglianze} and $h^1\left(\mathcal{I}_S(1,0)\right)=0$ we get the result.

Assume now $m\geq 2 $ and fix $H\in \vert \mathcal{O}_Y(\varepsilon_2) \vert$ containing $W $. By induction we get $\delta(2S,W)=\delta(2S,H) $. Since $H $ is a divisor of $ Y$ and $h^1(\mathcal{I}_S(0,1)) =0$ we get the result by applying the base case of \ref{a1disuguaglianze}.

Assume now $Y $ has $k\geq3 $ factors, i.e. $Y\cong W \times Y' $ where $Y' $ is a multiprojective space with at least two factors. Let $\mathbb{P}^{n_k} $ be the last factor of $Y $, again we will show the result by induction on $n_k\geq 1 $. If $n_k=1 $, one can always find $M\in \vert\mathcal{O}_Y(\varepsilon_k)  \vert $ containing $ W$ and by induction we get $\delta(2S,W)=\delta(2S,M) $. We remark as before that the Segre embedding of $ \nu(S)$ is linearly independent if and only if $\nu_k(S) $ is linearly independent and this is equivalent to say that $h^1(\mathcal{I}_S(\hat{\varepsilon}_k))=0 $. Since $M=\pi_k^{-1}(o) $, for some $o \in \mathbb{P}^1 $ we get the result by applying \ref{a1disuguaglianze}.

Assume now $n_k\geq2 $, and take some $M\in \vert \mathcal{O}_Y(\varepsilon_k) \vert $ containing $W $. By induction we get $\delta(2S,W)=\delta(2S,M) $, since $h^1(\mathcal{I}_S(\hat{\varepsilon}_k))=0 $ and $M $ is a divisor of $Y $ we get $\delta(2S,M)=\delta(2S,Y) $ by \ref{a1disuguaglianze}. 
\end{proof}

\begin{lemma}\label{a6}
Let $Y:= \PP^{n_1}\times \PP^{n_2}$ and $Y'\subseteq Y$ with $Y' := \PP^{m_1}\times \PP^{m_2}$ for some $m_i>0$. Let $S\subset
Y'$ be a finite subset such that $Y'$ is the minimal multiprojective space containing $S$ and suppose that both $\pi _{1|S}$ and $\pi _{2|S}$ are
injective and both $\pi _1(S)$ and $\pi _2(S)$ are linearly independent. Then $m_1=m_2=\#S -1$ and $h^1\left(Y',\Ii
_{(2S,Y')}(1,1)\right) =h^1\left(Y,\Ii
_{(2S,Y)}(1,1)\right)$.
\end{lemma}

\begin{proof}
Since $ \pi_i(S)$ is linearly independent and $ Y'$ is the minimal multiprojective space containing $ S$, then $m_1=m_2=\#S-1 $. Moreover since $ h^0\left(\Ii_S(1,0)\right)=h^0\left(\Ii_S(0,1)\right)=0$, then $h^1\left(\Ii _S(1,0)\right)=h^1\left(\Ii_S(0,1)\right) =0$. To conclude it is sufficient to use the proof of part \ref{a1disuguaglianze} of Lemma \ref{a1}.
\end{proof}

We recall here the Horace Differential Lemma   (\cite{ah1,ah4}, see also \cite{bccgo}).

\begin{lemma}[Horace Differential Lemma \cite{ah1,ah4}]\label{diff1}
Let $M$ be an integral projective variety, $D$ an integral effective Cartier divisor of $M$ and $\Ll$ a line bundle on $M$
such that
$h^i(\Ll) =0$ for all $i>0$ and $h^1(\Ll(-D))=0$. Set $n:= \dim M$. Let $Z\subsetneq M$ be a closed subscheme. Suppose
$h^1\left(M,\Ii_{\Res_D(Z)}\otimes
\Ll (-D)\right) =0$ and $h^1\left(D,\Ii _{Z\cap D,D}\otimes \Ll_{|D}\right)=0$. Fix $i\in \{0,1\}$. To prove that a general union $A$ of $Z$ and
one double point satisfies $h^i\left(\Ii_A\otimes \Ll\right)=0$ it is sufficient to prove that  $h^i\left(\Ii _{\Res_D(Z\cup (2o,D))}\otimes
\Ll(-D)\right)=0$ and $h^i\left(D,\Ii _{(Z\cap D)\cup \{o\}}\otimes \Ll _{|D})\right)=0$, where $o$ is a general point of $D$. 
Since $o$ is general in $D$,
$h^1\left(D,\Ii _{(Z\cap D)\cup \{o\}}\otimes \Ll _{|D})\right)=0$ if and only if $h^1\left(D,\Ii _{Z\cap D,D}\otimes \Ll_{|D}\right)=0$
and $h^0\left(D,\Ii _{Z\cap D,D}\otimes \Ll_{|D}\right)>0$.
The same trick works for a general union of $Z$ and finitely many double points.
\end{lemma}

\begin{proposition}\label{1lu1}
Write $Y =\PP^{n_1}\times Y_1$ as in Notation \ref{Notation:Yi} and take a closed subscheme $Z_1\subset Y_1$.
Then $$\dim \langle \nu (Z_1)\rangle =(n_1+1)(\dim\langle\nu_1(Z_1)\rangle+1)-1.$$
\end{proposition}

\begin{proof}
By assumption $h^0\left(Y_1,\Ii _{Z_1,Y_1}(1,\dots ,1)) = h^0(\Oo_{Y_1}(1,\dots ,1)\right) -\dim\langle\nu_1(Z_1)\rangle-1$. The K\"{u}nneth formula gives
$$h^0\left(\Ii_{Z_1}(1,\dots ,1)\right) =(n_1+1)\left(h^0(\Oo_{Y_1}(1,\dots ,1)) -\dim\langle\nu_1(Z_1)\rangle-1\right)-1.$$ Since $h^0\left(\Oo_Y(1,\dots ,1)\right) =(n_1+1)h^0\left(\Oo
_{Y_1}(1,\dots ,1)\right)$, we get the lemma.
\end{proof}

\begin{proposition}\label{1lu2}
Fix a finite set $S\subset Y$. Assume that there exist and index $i\in \{1,\dots ,k\}$ for which the projection $\eta _{i|S}: Y \to Y_i$ is injective. If $\delta (2\eta
_1(S),Y_i)=0$, then also $\delta (2S,Y)=0$.
\end{proposition}
\begin{proof}
With no loss of generality we may assume $i=1$. Set $S':= \eta _1(S)$, $s:=\#S$ and $m:= n_2+\cdots +n_k =\dim Y_1$.  The submersion
$\eta _1: Y\to Y_1$ has the property that
$\eta _1(i^{\ast}(\Oo_{Y_1}(1,\dots ,1)))
\cong
\Oo_Y(\hat{\epsilon} _1)$ and this isomorphism induces an isomorphism of global section. By assumption $2S'$ imposes
$s(n-n_1+1)$ independent conditions to $H^0(Y_1,\Oo_{Y_1}(1,\dots ,1))$. Thus the scheme $\eta _1^{-1}(2S')$ imposes
$s(n-n_1+1)$ independent conditions to $H^0(\Oo_Y(\hat{\epsilon}_1))$. The scheme $\eta _1^{-1}(S')$ is the union of $s$
disjoint varieties isomorphic to $\PP^{n_1}$ and embedded by $\nu$ as linear spaces and $\eta _1^{-1}(2S')$ is the union of
the first infinitesimal neighborhoods of $\PP^{n_1}$ in $Y$. By Lemma
\ref{1lu1} the scheme
$(2\eta _1(S'),Y)$ imposes $s(n_1+1)(m+1)$ independent conditions to $H^0(\Oo_Y(1,\dots ,1))$, i.e. the $s$ connected
components of $\eta _1^{-1}(2S')$ spans linearly independent linear spaces. For each $o\in S$ the scheme $\eta _i^{-1}(2o') = 2\eta _i^{-1}(o')$, $o':= \eta _1(o)$,
contains the double point $(2o,Y)$. In the Segre embedding the scheme $\nu((2o,Y))$ gives $\dim Y+1$ independent conditions. Since the $s$
subspaces spanned by the connected components of $\nu(\eta _1^{-1}(2S'))$ are linearly independent, $\nu(2S,Y)$ is linearly independent, i.e. $\delta (2S,Y)=0$.
\end{proof}

\section{The examples}\label{section:examples}

The following examples will be crucial for the main theorem.

\subsection{Example: Products of $\mathbb{P}^1$'s with at most one $\mathbb{P}^2$ has non empty 3-rd Terracini locus if and only if it envolves at least 4 factors}\label{ex1}

\begin{proposition}
Let $ Y=\mathbb{P}^m\times (\mathbb{P}^1)^{k-1}$ for some $ k\geq 3$, with $ m\in \{ 1,2\}$.
Define $S:=\{ a,b,c\} \subset Y $ be such that
\begin{align*}
&a:= (a_1,u_2,\dots ,u_k), b:=(b_1,u_2,\dots ,u_k),
c :=(c_1,\dots,c_k), \mbox{ with } \\
&a_1,b_1,c_1\in \PP^m \mbox{ such that } a_1\neq b_1 \mbox{ and } u_i\neq c_i \mbox{ for all } i>1.
\end{align*}
Moreover if $ m=2$ assume also $ \dim \langle \pi_1(S)\rangle=2$.
Therefore 
    $S\in \TT(Y,3) $ if and only if $k\geq 4$. 
    \end{proposition}
    
\begin{figure}[H]
    \centering
    \includegraphics[width=15cm]{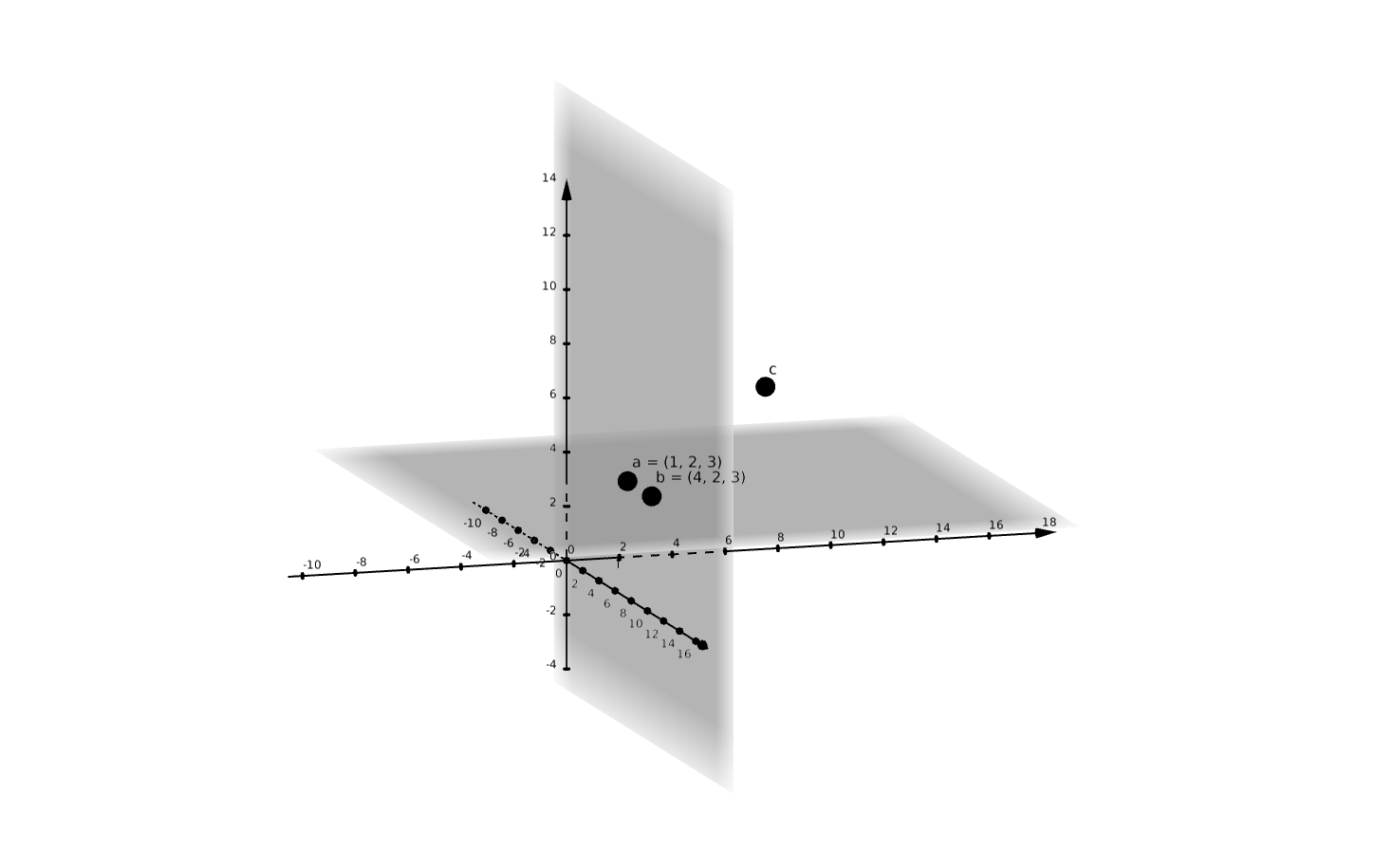}
    \caption{Picture of Proposition \ref{ex1} with $m=1$.}
    \label{ex1a:figure}
\end{figure}
\begin{figure}[H]
\centering
    \includegraphics[width=15cm]{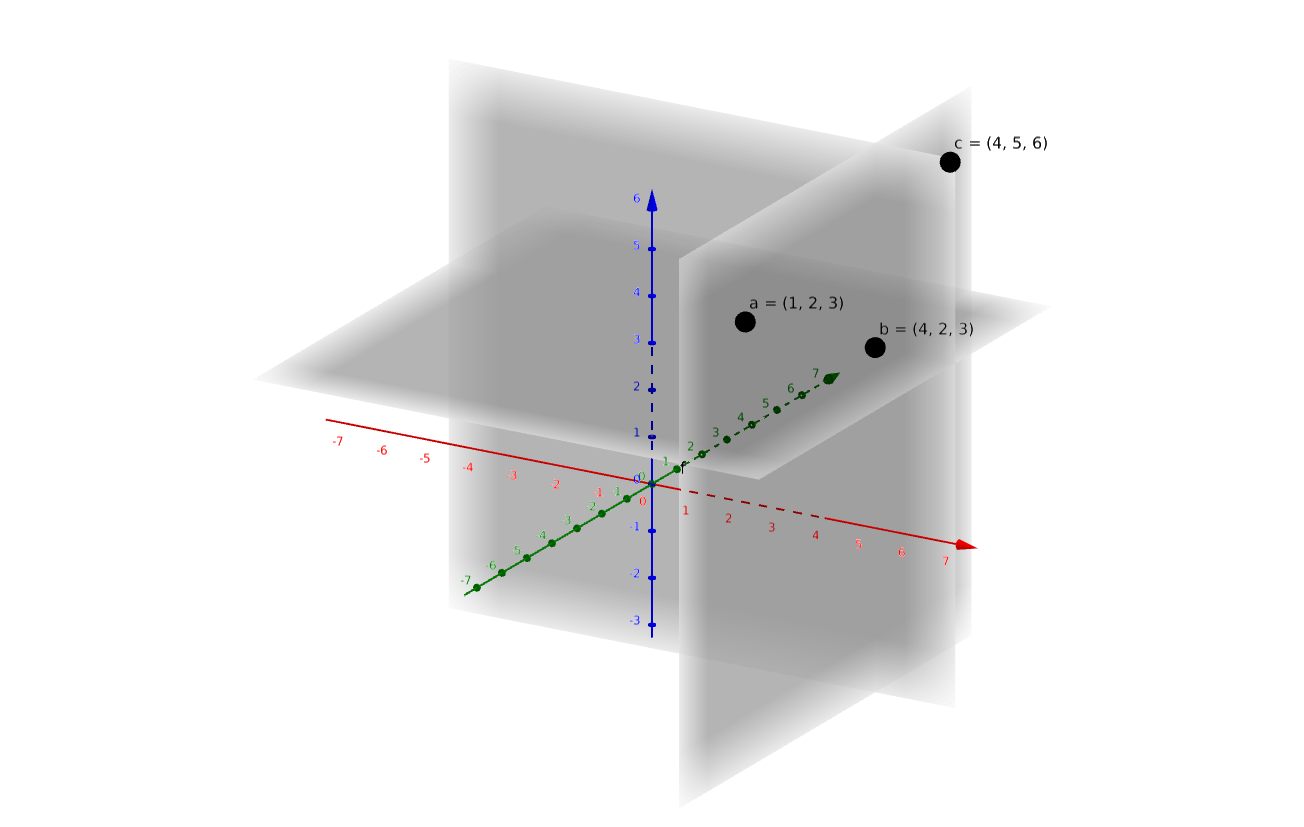}
    \caption{Pseudo-picture of Proposition \ref{ex1} with $m=2$ (the red axis is a $ \mathbb{P}^2$)}
    \label{ex1b:figure}
\end{figure}

\begin{proof}
We remark that since $ \#\pi_i(S)\geq 2$ for all $ i$'s and $\#\pi_1(S)=3 $ if $m=2 $, then $ Y$ is the minimal multiprojective space containing $ S$.
If we consider the subset $$S':=\{a,b \}$$ of $S$ we may apply Remark \ref{m1} and have that $\delta (2S,Y)\ge \delta (2S',Y)  $. Since $S'\subset \mathbb{P}^m\times \{ u_2\} \times \cdots \times \{ u_k \}\subset Y $, one can use case \ref{a1unfattore} of Lemma \ref{a1}, with $ W:=\mathbb{P}^m$, to get $$ \delta (2S',Y) = m+1.$$ Thus in order to see if $ S\in \mathbb{T}(Y,3)$, it suffices to understand whether  $$h^0\left(\mathcal{I}_{(2S,Y)}(1,\dots,1)\right)>0.$$ 

\begin{itemize}[leftmargin=+.2in]
    \item 
If $k\geq 4 $ then $h^0\left(\Oo_Y(1,\dots ,1) \right)=(m+1)2^{k-1} >3(m+k+1)=\deg(2S,Y) >0$ so we easily have that if $k\geq 4 $ then $S\in \mathbb{T}(Y,3) $.

\item

Let now $k=3 $. If we show that in this case none of the sets $S\subset Y$ as above belongs to $\TT(Y,3)$ we will be done.

Remark that by assumption $u_i \ne c_i$ for $i=2,3$.  To determine whether $ h^0\left(\Ii _{(2S,Y)}(1,1,1)\right) >0$ or not, we distinguish two cases depending on $m$ being equal to either 1 or  2.

\begin{enumerate}[leftmargin=+.2in, label=(\alph*)]
\item \label{ex1m=2}Assume $m=2$, i.e. $Y =\PP^2\times \PP^1\times \PP^1$.

Since $h^0(\Oo _Y(\epsilon _1))=3$, there exists $H\in|\Ii
_{\{a,c\}}(\epsilon 
_1)|$ and remark that  $$H\cong \PP^1\times \PP^1\times \PP^1.$$ Since $\langle \pi _1(S)\rangle =\PP^2$, then $H\cap S=\{a,c\}$.
Consider the residual exact sequence of $ S$ with respect to $ H$:
$$0\to \mathcal{I}_{\Res_{H}(2S,Y)}(0,1,1)\to \mathcal{I}_{(2S,Y)}(1,1,1) \to \mathcal{I}_{H\cap (2S,Y)}(1,1,1)\to 0.$$
Since $ H$ is smooth then $H\cap (2S,Y)= (2(S\cap H),H) = (2a\cup 2c,H)$ 
and the residue of $ (2S,Y)$ with respect to $ H$ is $$\Res _H(2S,Y )= \{a,c\}\cup (2b,Y).$$
Remark that $h^0\left(\Ii_{\Res_H(2S,Y)}(0,1,1)\right)=h^0\left(Y_1,\Ii _{\eta _1(\Res_H(2S,Y))}(1,1)\right)$. \\ 
Since
$\pi_i(a)=\pi_i(b) \neq \pi_i(c)$ for $i=2,3$, then $\eta_1(\Res_H(2S,Y)) =\eta _1(\{a,c\}\cup (2b,Y))=\eta _1(c) \cup (2\eta_1(b),Y_1)$.\\
In order to compute $h^0\left(Y_1,\mathcal{I}_{\eta _1(c) \cup (2\eta_1(b),Y_1)}(1,1)\right) $, we have to look at the hyperplanes of $\mathbb{P}^3 $  containing both   $\nu_1(\eta_1(c)) $ and $T_{\nu _1(\eta_1(b))}\nu _1(Y_1)$.
Note that the tangent space $T_{\nu _1(b)}\nu _1(Y_1)$ is the union of two lines through $\nu _1(b)$, i.e. the
image by
$\nu _1$ of the set of all $x\in Y_1$ with $\pi _2(x)=\pi _2(v)$ and the set of all $y\in Y_1$ with $\pi _3(y)=\pi _3(v)$. Thus, since $u_i\ne c_i$ for $i=2,3$, there are no such hyperplanes, hence $$h^0\left(Y_1,\Ii _{\eta _1(\Res_H(2S,Y)}(1,1)\right)=0.$$ So by the residual
sequence of $ S$ with respect to $H$ recalled above, it is sufficient to prove that $h^0\left(H,\Ii _{(2a\cup 2c,H)}(1,1,1)\right)=0$.

\smallskip

Since $\langle \pi_1(S)\rangle=\mathbb{P}^2  $ then $\pi _i(a) \ne \pi _i(c)$ for $i=1,2,3$. Since $H\cong \PP^1\times \PP^1\times \PP^1$ thus $\{a,c\}$ is in the open orbit of $\nu(\PP^1\times \PP^1\times \PP^1)$ for the action of $(\mathrm{Aut}(\PP^1))^3$ on $H$.
Since $\sigma _2(\nu (\PP^1\times \PP^1\times \PP^1)) =\PP^7$, then $$h^0\left(H,\Ii _{(2\{a,c\},H)}(1,1,1)\right)=0.$$ 

\item \label{ex1m=1}Assume $m=1$, i.e. $Y =(\PP^1)^3$. 

Fix $H\in |\Ii_a(\epsilon _3)|$. Since $u_i\ne c_i$ for $i=2,3$ and $H$ is smooth we have that $H\cap S =\{a,b\}$ and $\Res_H(2S,H) =\{a,b\}\cup (2c,Y)$. As in the last part of step \ref{ex1m=2}, we remark that $h^0(\Ii _{\Res_H(2S,Y)}(1,1,0))=h^0(Y_3,\Ii_{\eta_3(\{a,b\}\cup (2c,Y))}) $ and in order to compute it we have to look at the hyperplanes of $\mathbb{P}^3$ containing both $T_{\nu_3(\eta_3(c))}\nu_3(Y_3)$ and $\nu_3(\eta_3(\{a,b \}))$. \\ So $h^0(\Ii _{\Res_H(2S,Y)}(1,1,0)) =0$. Moreover, identifying $\nu (H)$ with a smooth quadric surface, by Proposition \ref{d1} we get $h^0(H,\Ii
_{(2S,Y)}(1,1,1))=0$ .
\end{enumerate}
\end{itemize}

Thus any set of points $ S\subset Y$ constructed as above is in the $3$-rd Terracini locus $  \TT(Y,3)$ if and only if $ k\geq 4$. 
\end{proof}

\subsection{Example: Products of $\mathbb{P}^1$'s with at most two $\mathbb{P}^2$'s}\label{a4.0}

\begin{proposition}
Let $Y:=\mathbb{P}^{n_1}\times \mathbb{P}^{n_2}\times (\mathbb{P}^1)^{k-2} $, where $n_1,n_2\in \{1,2 \}$ and  $k\geq 3$. Fix a line $L_1\subseteq \PP^{n_1}$, a line $L_2\subseteq \PP^{n_2}$.
Let $S:=\{o,u,v \} $ where
\begin{align*}
  u,v \in Y &\mbox{ such that } \pi_i(u)=\pi_i(v) \mbox{ for all } i>2 \mbox{ and } \langle \pi_j(u),\pi_j(v) \rangle=L_j \mbox{ for } j=1,2;\\
o \in Y& \mbox{ such that } \pi _i(o) \ne \pi_i(u) \mbox{ for all } i>2 \mbox{ and if } n_j=2 \mbox{ also assume } \pi _j(o)\notin L_j.
\end{align*}

Set $$Y':= L_1\times L_2\times \{\pi_3(u)\}\times \cdots \times \{ \pi_k(u) \}\subset Y.$$ 

Therefore
\begin{enumerate}[leftmargin=+.4in, label=(\roman*)]

\item \label{a4.0itemi}$2\le\delta(2S,Y)\le 5$.
\item \label{a4.0itemii} If $k\ge 4$ then $\delta (2S,Y) =2$.
\item \label{a4.0itemiii}If $k=3$ and $n_1=n_2=2$ then $ \delta(2S,Y)=2$ .
\item\label{a4.0itemiv}If $k=3$ and $n_1=n_2=1$ then $4\leq \delta(2S,Y)\leq 5$ and $h^0(\Ii_{(2S,Y)}(1,1,1)) >0$ if and only if $\pi _i(u)=\pi _i(o)$ and $\pi _h(v) = \pi _h(o)$ for some $i, h\in \{1,2\}$.   

\item\label{a4.0itemv} If $k=3$ and $\{n_1,n_2\} =\{1,2\}$ then $\delta(2S,Y) \ge 3$ and $h^0(\mathcal{I}_{(2S,Y)}(1,1,1))>0 $ if and only if $ \pi_2(o)\in \pi_2(S')$. 

\end{enumerate}
\end{proposition}

\begin{figure}[H]
\centering
    \includegraphics[width=15cm]{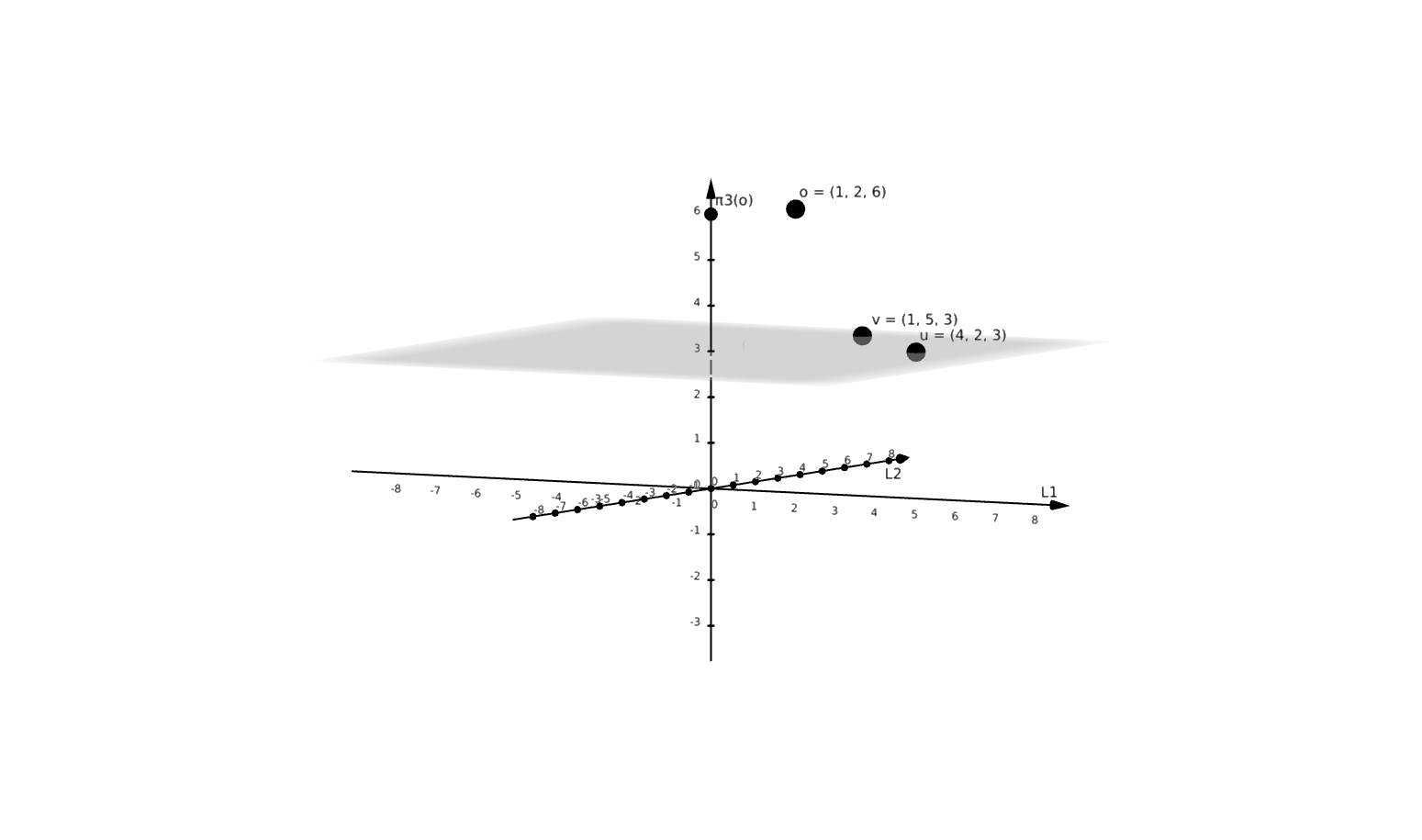}
    \caption{Picture of Proposition \ref{a4.0}.}
    \label{a4.0.figure}
\end{figure}

\begin{proof}
Remark that $S':=\{u,v\}\subset Y'$, and $Y'$ is actually the minimal multiprojective space containing $S'$ while $Y $ is the minimal multiprojective space containing $S $. 

Since
$\pi _k(o)\ne \pi _{k}(u)$, there is $H\in |\Oo_Y(\epsilon _k)|$ such that $o\in H$ and $S'\not\subset H$. Consider the residual exact
sequence
\begin{equation}\label{eqa5}
0 \to \Ii _{(2S',Y)\cup \{o\}}(\hat{\epsilon}_k)\to \Ii _{(2S,Y)}(1,\dots ,1) \to \Ii_{(2o,H),H}(1,\dots ,1)\to 0.
\end{equation}

Since $\Oo_H(1,\dots ,1)$ is very ample,  in order to prove \ref{a4.0itemi} it would be sufficient to prove that $h^1\left(Y,\Ii
_{(2S',Y)\cup
\{o\}}(\hat{\epsilon}_k)\right)\le 5$, but we prefer to  firstly do all cases from which the first assertion follows.

\medskip

Remark that $h^1\left(\Ii _{S'}(\epsilon _1)\right) =  h^1\left(\Ii _{S'}(\epsilon _2)\right) = 0$ because $\#\pi _i(S') =2$ for $i=1,2$.
By \eqref{eqa5} this implies that $h^1\left(\Ii _{(2S',Y)}(\hat{\epsilon}_k)\right)=2$ and hence $h^1\left(Y,\Ii _{(2S',Y)\cup \{o\}}(\hat{\epsilon}_k)\right)\le
3$.

\begin{enumerate}[leftmargin=+.2in]
\item[\ref{a4.0itemii}]  Assume $k\ge 4$.
We remark that one can identify
$$ h^1\left(Y',\mathcal{I}_{(2S',Y')}(1,\dots,1)\right)=h^1\left(Y,\mathcal{I}_{(2S',Y)}(1,1,0,\dots,0)\right)=h^1\left(Y,\mathcal{I}_{(2S',Y)}(\hat{\epsilon}_k)\right),$$
where the last equality follows since $\#\pi_i(S')=1 $ for all $i>2 $ and the previous equality follows from case  \ref{a1unfattore} of Lemma \ref{a1}. 
Moreover, since $\pi _i(o)\ne \pi _i(u)$, we have $h^1\left(\Ii _{(2S',Y)}(1,1,0,\dots ,0)\right) =h^1\left(\Ii _{(2S',Y)\cup
\{o\}}(\hat{\epsilon}_k)\right) $. 

If $n_1=n_2=1$, then $h^0 \left(\Ii _{(2S',Y)}(1,1,0,\dots ,0)\right)=0$, $h^0\left(\mathcal{O}_Y(1,1,0,\dots,0)\right)=4 $ and $ \deg(2S',Y)=6$. So $h^1\left(\Ii _{(2S',Y)}(1,1,0,\dots ,0)\right)=2 $.

If $\{n_1,n_2 \}=\{ 2,1\} $, one can use case \ref{a1disuguaglianze} of Lemma \ref{a1} with $W=(\mathbb{P}^1)^k $ and  get the same result because $h^1\left(\mathcal{I}_{S'}(\hat{\epsilon}_1)\right)=0 $. 

If $n_1=n_2=2 $, again by case \ref{a1disuguaglianze} of Lemma \ref{a1} applied to $ W=\mathbb{P}^{2}\times(\mathbb{P}^1)^{k-1}$ we get $h^1\left(\Ii _{(2S',Y)}(1,1,0,\dots ,0)\right)=2 $.

\item[\ref{a4.0itemiii}]  Assume $k=3$ and $n_1=n_2=2$. 

Consider the subgroup $G$ of $\mathrm{Aut}(\PP^2)\times \mathrm{Aut}(\PP^2)\times \mathrm{Aut}(\PP^1)$
fixing pointwise $Y'$. The action of $G$ on $Y$ has an open orbit $U$,
and $o\in U$. 
We remark that since $\#\pi_3(S')=1 $, one can identify $ h^1\left(Y,\mathcal{I}_{(2S',Y)}(1,1,1)\right)=h^1\left(Y,\mathcal{I}_{(2S',Y)}(\hat{\epsilon}_3)\right)$.
Moreover, since $\pi _3(o)\ne \pi _3(u)$, we have $h^1\left(\Ii _{(2S',Y)}(\hat{\epsilon}_3)\right) =h^1\left(\Ii _{(2S',Y)\cup
\{o\}}(\hat{\epsilon}_3)\right) $.

Thus to prove that $h^1\left(\Ii _{(2S',Y)\cup \{o\}}(\hat{\epsilon}_3)\right)=2$ and hence to prove part \ref{a4.0itemiii}, it is sufficient to observe that $h^0\left(W,\Ii _{(2A,W)}(1,1)\right) >0$, because not all
$3\times 3$ matrices have rank at most $ 2$, where we took $W=\mathbb{P}^2\times \mathbb{P}^2 $ and $A \subset W$ as a general subset of two distinct points.

\item[\ref{a4.0itemiv}]  Assume $k=3$ and $n_1=n_2=1$. 

Since $h^0\left(\Oo _Y(1,1,1)\right) =8$ and $\deg (2S,Y) =12$, we have $h^1\left(\Ii_{(2S,Y)}(1,1,1)\right) =4 +h^0\left(\Ii_{(2S,Y)}(1,1,1)\right)$, so $h^1\left(\Ii
_{(2S,Y)}(1,1,1)\right) \ge 4$. 
To conclude this case it is sufficient to show the following 
\begin{quote}
    \begin{claim}\label{a4.0claim} With the notation as above $h^0\left(\Ii _{(2S,Y)}(1,1,1)\right) >0$ if and only if $\pi _i(u)=\pi _i(o)$ for some $i\in \{1,2\}$ and $\pi _j(v) = \pi _j(o)$ for some $j\in \{1,2\}$. In this case $h^0\left(\Ii _{(2S,Y)}(1,1,1)\right)=1$ and $h^1\left(\Ii_{(2S,Y)}(1,1,1)\right) =5$.
    \end{claim}
    \medskip
    \begin{proof}
    Take $H\in |\Ii_{\{u\}}(\epsilon _3)|$. Since $\pi_3(u)=\pi_3(v)$, then $H\cap S =S'$. Since $H$
is smooth, $(2S,Y)\cap H= (2S',H)$ and $\Res _H(2S,Y) = S'\cup \{2o\}$. We identify $\nu (H)$ with a smooth
quadric surface $Q\subset \mathbb{P}^3$. 
Since a tangent plane to a smooth quadric surface $Q$ is tangent to $Q$ at a unique point, then we have the vanishing of $h^0\left(H,\Ii
_{(2S,Y)\cap H,H}(1,1,1)\right)$.

Consider the residual exact sequence of $ S$ with respect to $H$:
\begin{equation}\label{eqrr1}
0 \to \Ii _{S'\cup \{2o\}}(1,1,0)\to \Ii _{(2S,Y)}(1,1,1)\to \Ii_{(2S,Y)\cap H,H}(1,1,1)\to 0.
\end{equation}
Since $h^0\left(H,\Ii_{(2S,Y)\cap H,H}(1,1,1)\right)=0$, 
then  $$h^0\left(\Ii _{(2S,Y)}(1,1,1)\right)=h^0\left(\Ii _{S'\cup \{2o\}}(1,1,0)\right).$$
Moreover  $ h^0\left(\Ii _{S'\cup \{2o\}}(1,1,0)\right)=
h^0\left(Y_3,\Ii _{\eta _3(S')\cup (2\eta _3(o),Y_3)}(1,1)\right)$ and we can think of $\nu _3(Y_3)$ as a smooth quadric surface.
Since $T_{\nu _3(o)}\nu _3(Y)$ is a plane $h^0\left(Y_3,\Ii _{\eta _3(S')\cup (2\eta _3(o),Y_3)}(1,1)\right)\le 1$. 
\\
Now $h^0\left(Y_3,\Ii _{\eta _3(S')\cup (2\eta _3(o),Y_3)}(1,1)\right)= 1$ if and only if both $\nu _3(\eta _3(u))$
and $\nu _3(\eta _3(u))$ are contained in $\nu _3(Y_3)\cap T_{\nu _3(\eta_3(o))}(\nu _3(Y_3))$. 
We remark that $T_{\nu _3(\eta_3(o))}(\nu _3(Y_3))$ is the union of two lines through $\nu_3(\eta_3(o)) $, i.e. the
image by $\nu _3$ of the set of all $x\in Y_3$ with $\pi _1(x)=\pi _1(o)$ and the set of all $y\in Y_3$ with $\pi _2(y)=\pi _2(o)$. Hence Claim \ref{a4.0claim} is
just a translation of this observation.\end{proof} 
\end{quote}

\medskip

\item[\ref{a4.0itemv}] Assume $\{n_1,n_2\} =\{1,2\}$ and $k=3$. 

With no loss of generality we may assume $n_1=2$ and $n_2=1$. Since $h^0\left(\Oo
_Y(1,1,1)\right) =12$ and $\deg (2S,Y) =15$, we have $h^1\left(\Ii _{(2S,Y)}(1,1,1)\right)= 3 +h^0\left(\Ii_{(2S,Y)}(1,1,1)\right)$. Hence $\delta(2S,Y)\geq3 $.\\ We remark that
by assumption $\pi _1(o) \notin \langle \pi _1(u),\pi _1(v) \rangle$, $\pi _2(u)\ne \pi _2(v)$ and $\pi _3(o) \ne \pi_3(u)=\pi _3(v)$.\\ To conclude this case we have to show that $ h^0\left(\mathcal{I}_{(2S,Y)}(1,1,1)\right)>0$ if and only if $\pi_2(o)\in \pi_2(S') $.
\begin{itemize}[leftmargin=+.2in]
\item Assume $\pi _2(o) \in \pi _2(S')$. Without loss of generality we may assume that $\pi _2(u) =\pi _2(o)$. Since $ h^0\left(\mathcal{O}_Y(\epsilon_2)\right)=2$ then $|\Ii _o(\epsilon _2)| $ is a singleton. Set $\{H\}:= |\Ii _o(\epsilon _2)|$. Since $H\cong \PP^2\times \PP^1$ it is smooth, hence $(2S,Y)\cap H =(2\{o,u\},H)$ scheme-theoretically and $\Res _H(2S,Y) =(2v,Y)\cup \{o,u\}$. Remark that $h^0\left(\Ii _{(2v,Y)\cup \{o,u\}}(1,0,1)\right) =h^0\left(Y_2,\Ii _{(2\eta _2(v),Y_2)\cup \{\eta _2(o),\eta _2(u)\}}(1,1)\right)$. We remark that $Y_2\cong \PP^2\times \PP^1$ and  $h^0\left(\Oo_{Y_2}(1,1)\right) =6 =\deg ((2\eta _2(v),Y_2)\cup \{\eta _2(o),\eta _2(u)\})$. This last equality implies that
$$\qquad h^0\left(Y_2,\Ii _{(2\eta _2(v),Y_2)\cup \{\eta _2(o),\eta _2(u)\}}(1,1)\right) = h^1\left(Y_2,\Ii _{(2\eta _2(v),Y_2)\cup \{\eta _2(o),\eta _2(u)\}}(1,1)\right).$$ 
To show that $ h^0\left(Y_2,\Ii _{(2\eta _2(v),Y_2)\cup \{\eta _2(o),\eta _2(u)\}}(1,1)\right)>0$, we have to look at the hyperplanes of $\mathbb{P}^5\supset \nu_2(Y_2) $ that contain both the tangent space $T_{\nu _2(\eta _2(v))}\nu _2(Y_2)$ and the points  $\nu_2(\eta_2(\{ o,u\}))$. Remark that $T_{\nu _2(\eta _2(v))}\nu _2(Y_2) \cap \nu _2(Y_2)$ is the union of $2$ linear spaces containing $\nu _2(\eta _2(v))$, one of dimension $2$ and one of dimension $1$, spanning the $3$-dimensional projective space $T_{\nu _2(\eta _2(v))}\nu _2(Y_2)$. Since $\pi _3(u) =\pi _3(v)$, then $\nu _2(\eta _2(u))$ is a point of the $2$-dimensional irreducible component
of the tangent space $T_{\nu _2(\eta _2(v))}\nu _2(Y_2) \cap \nu _2(Y_2)$.\\ Thus   $h^0\left(\Ii _{\Res_H(2S,Y)}(1,0,1)\right) >0$. Hence, by the long cohomology exact sequence induced by the exact sequence of the residue of $ S$ with respect to $ H$, we get $h^0\left(\Ii _{(2S,Y)}(1,1,1)\right) >0$.

\item Assume $\pi _2(o)\notin \pi _2(S')$. Since $h^0\left(\mathcal{O}_Y(\epsilon_3)\right)=2 $ then $|\Ii _u(\epsilon _3)|$ is a singleton. Set $\{M\}:= |\Ii _u(\epsilon _3)|$. Since $M$ is smooth and $M\cap S =S'$, then $(2S,Y)\cap M = (2S',M)$ scheme-theoretically
and $\Res _M(2S,Y) = S'\cup (2o,Y)$. We have $h^0\left(\Ii _{S'\cup (2o,Y)}(1,1,0)\right) = h^0\left(Y_3,\Ii _{\eta _3(S')\cup (2\eta _3(o),Y_3)}(1,1)\right)$. Obviously $h^0\left(Y_3,\Ii _{(2\eta _3(o),Y_3)}(1,1)\right) =2$. Since $\eta _3(S)$ is in the open orbit of $S(Y_3,3)$ for the action of $\mathrm{Aut}(\PP^2)\times \mathrm{Aut}(\PP^1)$, $h^0\left(Y_3,\Ii _{\eta _3(S')\cup (2\eta _3(o),Y_3)}(1,1)\right)=0$. The set $S'$ is in the open orbit of $\mathrm{Aut}(M)$ for its action in $S(M,3)$. Since any $3\times 2$ matrix has rank at most $ 2$, we know that $\sigma _2(\nu (M)) =\PP^5$. Thus $h^0\left(H,\Ii _{(2S,Y)\cap H,H}(1,1,1)\right) =0$. The residual exact sequence of $ S$ with respect to $M$ gives $h^0\left(\Ii _{(2S,Y)}(1,1,1)\right) =0$.
\end{itemize}
In summary if $k=3 $ and $\{n_1,n_2 \}=\{1,2 \}$, then $S\in \TT(Y,3)$ if and only if $\pi _2(o)\in \pi _2(S')$.
\end{enumerate}
\vspace{-0.4cm}
\end{proof}

\section{Main theorem}\label{section:main}

In this section we prove the main theorem of the present paper.

\begin{remark}\label{azionesuP1}
Let  $Y=(\mathbb{P}^1)^k  $, for some $k\geq 2 $. Given any two subset $S,S'\in S(Y,3) $ such that $\# \pi_i(S)=\# \pi_i(S')=3 $ for all $ i$'s, one can always find $ f \in (\mathrm{Aut}(\mathbb{P}^1))^k$ such that $f(S)=f(S') $. Since $ Y$ is the minimal multiprojective space containing both $S $ and $ S'$, then $S\in \TT(Y,3)$ if and only if $S'\in\TT(Y,3)$.
\end{remark}

\begin{lemma}\label{lu2}
Let $Y=\mathbb{P}^1\times \mathbb{P}^1\times \mathbb{P}^1 $ and let $ S:=\{u,v,o \} \in S(Y,3)$ such that $Y $ is the minimal multiprojective space containing $S $. Then $S \in \TT(Y,3) $ if and only if there exist $h\in \{1,2,3\}$  and   $i,j\in \{1,2,3\}\setminus \{h\}$, with $i<j$ such that 
\begin{align*}
    \pi_h(u)=\pi_h(v)\neq &\pi_h(o),\, \pi_i(o)=\pi_i(u) \mbox{ and } \pi_j(o)=\pi_j(v),\\
  \mbox{where both }&  \pi_i(u)\neq \pi_i(v) \mbox{ and } \pi_j(u)\neq \pi_j(v).
\end{align*}

\end{lemma}

\begin{proof}
Up to a permutation of the index $h \in \{1,2,3\}$ we may assume $h=3$.\\
Take $S\in S(Y,3)$ and let $X:=\nu(Y) $. If $\#\pi_i(S)=3$ for all $i$'s, since $\dim \sigma_3(X)=7$ (e.g. \cite{CGG07}), we have $h^0\left(\Ii_{(2S,Y)}(1,1,1)\right)=0$. Hence, by Remark \ref{azionesuP1}, we have that $S \notin \TT(Y,3) $.  Now assume $\#\pi _i(S)\le 2$ for some $i$. Remark that since $Y$ is the minimal multiprojective space containing $ S$ then $\#\pi_i(S)=2 $. We distinguish different cases depending on the number of indices $i \in \{1,2,3 \}$ for which $\#\pi_i(S)=2$. 
\begin{itemize}[leftmargin=+.2in]
    \item If there exists only an index $i$ such that $\#\pi_i(S)=2 $ then $ S$ is as in Example \ref{a4.0} and by case \ref{a4.0itemiv} of Example \ref{a4.0} we know that $ h^0\left(\Ii_{(2S,Y)}(1,1,1)\right)=0$. 
    \item If $\#\pi_i(S)=2 $ for two indices, then $S$ is as in Example \ref{a4.0} or as in Example \ref{ex1}. For both cases we have $ h^0\left(\Ii_{(2S,Y)}(1,1,1)\right)=0$.
    \item  Finally, if $\#\pi_i(S)=2 $ for all $i\in \{1,2,3 \} $, then $ S$ is as in Example \ref{a4.0} and by case \ref{a4.0itemiv} of Example \ref{a4.0} we get that $ h^0\left(\Ii_{(2S,Y)}(1,1,1)\right)=1$. 
\end{itemize}
\vspace{-5mm}
\end{proof}

\begin{remark}\label{lu1}
Fix $Y=(\mathbb{P}^1)^4 $ and let $ A\subset Y$ be a general subset of three distinct points. Since the $ 3$-rd secant variety of $ X:=\nu(Y)\subset \mathbb{P}^{15}$ is defective with defect $1 $ (cf. \cite{CGG, CGG05}), then $\dim(\sigma_3(X))=13$ and $ h^0\left(\mathcal{I}_{(2A,Y)}(1,1,1,1) \right)=2$, so by looking at the restriction exact sequence we get $ h^1\left(\mathcal{I}_{(2A,Y)}(1,1,1,1)\right)=1$. By the semicontinuity theorem for cohomology (cf. \cite[Ch. III \S 12]{hart}) $h^1\left(\Ii_{(2S,Y)}(1,1,1,1)\right)\ge 1$ and $h^0\left(\Ii_{(2S,Y)}(1,1,1,1)\right) \ge 2$ for all $S\in S(Y,3)$. Moreover we remark that $Y$ is the minimal multiprojective space containing $S\in S(Y,3) $ if and only if $\pi_i(S)\ge 2$ for all $i=1,2,3,4$.

Thus $S(Y,3)\subset \mathbb{T}_1(Y,3) $ and the $3$-rd Terracini locus $ \mathbb{T}(Y,3)$ contains all subsets $ S\in S(Y,3)$ such that $\pi_i(S)\ge 2$ for all $i$.
\end{remark}

\begin{remark}\label{o1}
Let $Y=(\mathbb{P}^1)^k  $ with $k\geq5 $. 
By \cite[Theorem 2.3]{CGG} we know that $\dim (\sigma_3(X))=3k+2 $. By looking at the restriction exact sequence of $2S $ with respect to $ Y$, we get that $h^1\left(\mathcal{I}_{(2S,Y)}(1,\dots,1)\right) =0$. 
So a general $S\subset Y $ with $\# S=3 $ is not in the $3$-rd Terracini locus $\mathbb{T}(Y,3) $.
Thus by Remark \ref{azionesuP1}, for all $S \in S(Y,3) $ such that $\# \pi_i(S)=3 $ then $S \notin \mathbb{T}(Y,3) $.
\end{remark}

\begin{lemma}\label{o2}
 Let $Y =(\PP^1)^k$ with $k\ge 5$. Fix $ S:=\{ a,b,c\} \in S(Y,3)$ such that $ Y$ is the minimal multiprojective space containing $S $. Assume that there are at least $k-2 $ indices $ i$'s for which $ \pi_i(a)=\pi_i(b)$. Then $S $ is either as in Example \ref{a4.0} or as in Example \ref{ex1}.
\end{lemma}

\begin{proof}
 Define $E:= \{i\in \{1,\dots ,k\}\mid \pi_i(a)=\pi_i(b)\}$, by assumption $ \# E \geq k-2 $ and since $a\neq b $ then $\#E\leq k-1 $. By permuting the factors of $ Y$ if necessary, one can always assume that $ E$ contains the last $k-2 $ indices and that the index $1 \notin E$. If $ 2 \notin E$ then $S $ is constructed as in Example \ref{a4.0} with $n_1=n_2=1 $, else $ S$ is as in Example \ref{ex1} where we took $m=1 $. 
\end{proof}

\begin{lemma}\label{o3}
Let $Y =(\PP^1)^k$ with $k\ge 5$. Fix $S\in S (Y,3)$ such that $Y$ is the minimal multiprojective subspace containing $S$. If $S\in \TT(Y,3)$ then $S$ is either as in Example \ref{a4.0} or as in Example \ref{ex1}.
\end{lemma}

\begin{proof}
Write $S:=\{u,v,z \} $. Since $ S\in \mathbb{T}(Y,3)$, by Remark \ref{o1} we may assume that $ \pi_i(u)=\pi_i(v)$ for at least an $ i \in \{ 1,\dots,k\}$.
With no loss of generality we may assume $i=1$. Since $h^0\left(\mathcal{O}_Y(\epsilon_i)\right)=2$ for $i=1,2$, both $ |\Ii_u(\epsilon _1)|$ and $|\Ii_z(\epsilon _2)|$ are singletons. Set $\{H\}:= |\Ii_u(\epsilon _1)|$
and $\{M\}:= |\Ii_z(\epsilon _2)|$. Since $v\in H$, then $S\subset H\cup M$. Moreover, since $Y$ is the minimal multiprojective space containing
$S$, then $z\notin H$ and $\#(S\cap M)\le 2$.
\begin{quote}
\begin{claim}\label{claimo3}
$h^1\left(\Ii_S(0,0,1,\dots ,1)\right) =0$ unless $S$ is  either as in Example \ref{a4.0} or
as in Example \ref{ex1}.
\end{claim}
\begin{proof}
Call $\eta : Y\to (\PP^1)^{k-2}$ the projection onto the last $k-2$ factor of $Y$ and set $Y':=
(\PP^1)^{k-2}$. \\ Assume $h^1\left(\Ii_S(0,0,1,\dots ,1)\right)>0$. Therefore  either $\eta _{|S}$ is not injective or
$\#\eta (S) =3$ and $h^1\left(Y',\Ii_{\eta(S)}(1,\dots ,1)\right)>0$. In the first case $S$ is either as in Example \ref{a4.0} or as in
Example \ref{ex1} by Lemma \ref{o2}. In the second case by \cite[Lemma 4.4]{BBS} there is $i\in \{3,\dots ,k\}$ such that $\#\pi
_h(S)=1$ for all $h\in \{3,\dots ,k\}\setminus \{i\}$ contradicting the minimality of $Y'$ for $\eta(S)$, which is a
consequence of the minimality of $Y$ for $S$.
\end{proof}
\end{quote}

Assume by contradiction that $S$ is neither as in Example \ref{a4.0} nor
as in Example \ref{ex1}.
\begin{enumerate}[leftmargin=+.2in, label=(\alph*)]
    \item\label{o3casoa} Assume $\#(S\cap M)=1$, i.e. $S\cap (H\cap M)=\emptyset$. So $S$ is contained in the smooth part of $H\cup M$ and $\Res_{H\cup M}(2S) =S$. Since $ S$ is not as in one of the examples, by Claim \ref{claimo3} we get $h^1\left(\mathcal{I}_{\Res_{H\cup M }(2S,Y)}(0,0,1,\dots,1)\right)=h^1\left(\mathcal{I}_S(0,0,1,\dots,1)\right)=0 $. Moreover, by the restriction exact sequence of $ S$, we get $ h^0\left(\mathcal{I}_S(0,0,1,\dots,1)\right)=2^{k-2}-3$. Since by assumption $S\in \mathbb{T}(Y,3) $, then $h^0\left(\mathcal{I}_{(2S,Y)}(1,\dots,1)\right)>0 $ and more precisely $ h^0\left(\mathcal{I}_{(2S,Y)}(1,\dots,1)\right)\geq 2^k-3(k+1)$, where $k\geq 5 $. Thus the residual exact sequence of $H\cup M $ gives
$h^1\left(H\cup M, \Ii _{(2S,H\cup M),H\cup M}(1,\dots ,1)\right)>0$. 
Since  by assumption $S\cap (H\cap M)=\emptyset$, the scheme theoretic intersection 
$(2S,H\cup M) = (2u,H)\cup (2v,H)\cup (2z,M)$. 

Denote by $G$ the set of all $g\in (\mathrm{Aut}(\PP^1))^k$ acting as the identity on the
last $k-1$ factors of $Y$; we remark that the elements of $ G$ are $3$-transitive on the first factor. Let $G_u$ the subgroup of $G$ fixing also the
first component $\pi _1(u)$ of $u\in S$. Hence, since we assumed $ \pi_1(u)=\pi_1(v)$, any $g\in G_u$ fixes both $u$ and $v$. Obviously $h^1\left(H\cup M,\Ii_{(2u,H)\cup (2v,H)\cup
(2z,M),H\cup M}(1,\dots ,1)\right) = h^1\left(H\cup M,\Ii_{(2u,H)\cup (2v,H)\cup
(2g(z),M),H\cup M}(1,\dots ,1)\right)$ for all $g\in G_u$.
Thus it is sufficient to find a contradiction for a single $z'\in
M\setminus H\cap M$ with $\pi _i(z')=\pi _i(z)$ for $i>1$.

We may specialize $z $ by considering a general $o \in H\cap M $. So it is sufficient to work on $ H$ rather than $ H\cup M$. Denote by $Z:=(2u,H)\cup (2v,H)$ and call $ A$ the union of $ Z$ and the double point $(2o,H\cap M) $. We want to use the Differential Horace Lemma with $ H\cap M$ as a divisor of  $H$ (cf. Lemma \ref{diff1}). We remark that $Z\subset H $ satisfies  the assumptions of the Differential Horace Lemma, i.e. both 
$$h^1\left(H,\Ii_{\Res_{H\cap M}(Z)}\otimes
\Ll (-H\cap M)\right) =0$$
$$\hbox{and }h^1\left(H\cap M,\Ii _{Z\cap (H\cap M),H\cap M}\otimes \Ll_{|H\cap M}\right)=0.$$ Indeed the latter is trivially zero since by assumption $\#S\cap M=1 $. The former is zero since $ H$ is the minimal multiprojective space containing $Z $ and by Proposition \ref{d1} we know that $\mathbb{T}(H,2)=\emptyset $ and in particular that $\delta(2\{u,v\},H)=0 $. Thus by Lemma \ref{diff1}, in order to show that $h^1\left(H,\mathcal{I}_{A}(1,\dots,1)\right)=0 $, it suffices to show that both
$$h^1\left(H\cap M, \mathcal{I}_{(Z\cap (H\cap M))\cup \{o \} }(1,\dots,1)\right)=0$$ 
$$\hbox{and }h^1\left(H,\mathcal{I}_{\Res_{H\cap M}(A)}(1,0,1,\dots,1)\right)=0.$$
Clearly since $ (Z\cap (H\cap M))\cup \{o \}=\{ o\} $ then $h^1\left(H\cap M, \mathcal{I}_{(Z\cap (H\cap M))\cup \{o \} }(1,\dots,1)\right) $ is trivially zero. The second equality follows from \eqref{eqm1} of Remark \ref{m1} since we already pointed out that $\delta(2\{u,v\},H)=0$.

\item\label{o3casob} Assume $\#(S\cap M)=2$. Taking $\{M_i\}  =|\Ii_z(\epsilon _i)|$ for $i=3,\dots ,k$ and applying step \ref{o3casoa} to $H\cup
M_i$, we see that it is sufficient to handle the case with $\#\pi _i(S)=2$ for all $i$.

Write $S =\{a,b,c\}$. Set $a_i:= \pi _i(a)$, $b_i:= \pi _i(b)$ and $c_i:= \pi _i(c)$. Since $\mathrm{Aut}(\PP^1)$ is 2-transitive, by composing with an element of $\mathrm{Aut}(\PP^1)^k$, we may assume $\pi _i(S) =\{\alpha,\beta\}$ for all $i$. 

Without loss of generality we may also assume $a_i =\alpha$ for all $i$.
Thus $\beta \in \{b_i,c_i\}$ for all $i$. Moreover, since $S$
is neither as in Example \ref{a4.0} nor as in Example \ref{ex1},  for all $A\subset S$ with $\#A=2$ then $\#\pi _i(A)=2$ for at least $3$ indices $i$'s.

We define the maximum number of common components that any two points of $ S$ can have as 
$$t:=\max\{ \#I\subset \{1,\dots,k \} \, \vert \, \exists A \subset S \mbox{ with }\# A=2 \mbox{ such that } \;\forall i \in I \; \pi_i(A)=1 \}. $$
By relabeling if necessary, we may assume that $\{ a,b\} $ is one of the subsets of $ S$ reaching such $ t$.  
By assumption $t \le k-3$.

We distinguish different cases depending on the integer $k\geq 5 $. In particular, for $k=5,6 $, we will get to a contradiction with the assumption $ \delta(2S,Y)>0$ by direct computation with Macaulay2 (cf. \cite{M2}).
\begin{enumerate}[leftmargin=+.2in, label=(\roman*)]
    \item\label{o3casobitem1}  Assume $k=5$. So $t\leq 2 $ and since $\#\pi_i(S)=2$ for all $i$ and $k>3 $ then $t =2$. 
Permuting the factors of $Y$ we may assume $b_i =\alpha$ for $i=1,2$ and $b_i =\beta$ for $i=3,4,5$.  Since $\#\pi _i(S)=2$ for all $i$, then $c_1
=c_2 =\beta$. Since $a$ and $c$ have at most $2$ common components, then $c_i=\alpha$ for $i=3,4$ and $c_5=\beta$.
Thus $S=\{a,b,c\}$ is such that \begin{equation*}
    a=(\alpha,\alpha,\alpha,\alpha,\alpha), \;b=(\alpha,\alpha,\beta,\beta,\beta),\; c=(\beta,\beta,\alpha,\alpha,\beta)
\end{equation*} and up to a permutation of the elements of $ S$ and of the factors of $Y$, there is a unique such $ S$.\\ By direct computation one can see that $h^0\left(\mathcal{I}_{(2S,Y)}(1,1,1,1,1)\right)=14 $ and consequentially $h^1\left(\mathcal{I}_{(2S,Y)}(1,1,1,1,1)\right)=0$ contradicting the assumption. 

\item\label{o3casobitem2}Assume $k=6$. We have $t \le 3$.  Moreover, since $\#\pi_i(S)=2$ for all $i$, then $t \ge 2$. We distinguish two different cases in dependence on the value $t \in \{ 2,3 \} $.

Assume $t =3$.  Permuting if necessary the factors of $Y$, we may assume $b_1=b_2=b_3=\alpha$ and $b_4=b_5=b_6 =\beta$.
Thus since $ \#\pi_i(S)=2$ for all $i $'s, then $c_1=c_2=c_3=\beta$. 
Moreover $c $ and $ a$ can have $ 2$ or $3$ common components.
In the first case $S=\{a,b,c\}$ is such that \begin{equation*}
   a=(\alpha,\alpha,\alpha,\alpha,\alpha,\alpha),\;b=(\alpha,\alpha,\alpha,\beta,\beta,\beta),\; c=(\beta,\beta,\beta,\beta,\alpha,\alpha).
\end{equation*} 
In the second case $S=\{a,b,c\}$ is such that \begin{equation*}
    a=(\alpha,\alpha,\alpha,\alpha,\alpha,\alpha), \;b=(\alpha,\alpha,\alpha,\beta,\beta,\beta),\; c=(\beta,\beta,\beta,\alpha,\alpha,\alpha).
\end{equation*} We remark that up to permuting the factors of  $Y$ and relabeling the elements of $ S$, these are the only cases for $ t=3$. As before, by direct computation, one gets for both cases $\delta(2S,Y)=0 $ contradicting the assumption. 

Assume $t =2$. Permuting the factors of $Y$ we may assume $b_1=b_2=\alpha$ (and hence $c_1=c_2=\beta$) and
$b_3=b_4=b_5=b_6 =\beta$. Since $\#\{a_i,c_i\}=\#\{b_i,c_i \}=1$ for at most $2$ indices, among the set $\{3,4,5,6\}$ exactly $2$ $i$'s have $c_i=\beta$, while the other ones have $c_i=\alpha$. 
Thus $S=\{a,b,c\}$ is such that \begin{equation*}
     a=(\alpha,\alpha,\alpha,\alpha,\alpha,\alpha),\;b=(\alpha,\alpha,\beta,\beta,\beta,\beta),\; c=(\beta,\beta,\beta,\beta,\alpha,\alpha).
\end{equation*} 
Up to relabeling the points of $S$ and a permutation of the
factors of $Y$, there is a unique such $S$. By direct computation one gets $h^0\left(\mathcal{I}_{(2S,Y)}(1,1,1,1,1,1)\right)=20$, so $\delta(2S,Y)=0 $ contradicting the assumption.

\item Now assume $k\ge 7$.
Exchanging if necessary the names of the points of $S$ we may assume
$\pi_1(a)=\pi_1(b)=a_1$ and hence $\pi_1(c)=c_1\ne a_1$. For any $t\in
\PP^1$ set $S_t:= \{a,b,c_t\}$, where $\pi _1(c_t):=t$ and $\pi _i(c_t): =c_i$ for all $i>1$. Since any two of
points of $S$ differ in at least $3$ coordinates, $\#S_t =3$ for all $t$. Since $\mathrm{Aut}(\PP^1)$ is $3$-transitive, for
each $t\in \PP^1\setminus \{a_1\}$ there is $g_t\in (\mathrm{Aut}(\PP^1)^k)\subset \mathrm{Aut}(Y)$ such that $g_t(S_t)=S$.
Thus $\delta (2S) = \delta (2S_t)$ for all $t\in \PP^1\setminus \{a_1\}$. 
By the semicontinuity theorem for cohomology it is
sufficient to prove $\delta (2S_{a_1},Y)=0$.

To show that $\delta (2S_{a_1},Y)=0$, we proceed by induction on the integer $n:=k-7 $.

Assume $ n=0$, i.e. $ k=7$. Since $h^0(\Oo_Y(\epsilon_1))=2$, $|\Ii_a(\epsilon _1)|$ is a singleton. Set
$\{H\}:= |\Ii_a(\epsilon _1)|$, so $H\supset S_{a_1}$ by definition. Since any two points of $S$ differs in at
least $3$ coordinates, by \cite[Lemma 4.4]{BBS} we know that $h^1\left(H,\Ii_{S_{a_1}}(\hat{\epsilon}_1)\right) =0$. By case \ref{a1disuguaglianze} of Lemma \ref{a1} we know that
$\delta (2S_{a_1},Y) =\delta (2S_{a_1},H)$. In item \ref{o3casobitem2} we proved that for any subset $ S\subset (\mathbb{P}^1)^6$ of three points such that any two of them have at least $ 3$ distinct components, then $\delta(2S,(\mathbb{P}^1)^6)=0 $. Thus $ \delta (2S_{a_1},H)=0$, and hence $ \delta (2S_{a_1},Y) =0$.

Assume now $n>0 $, i.e. $ k>7$. As before, we set
$\{H\}:= |\Ii_a(\epsilon _1)|$, so $H\supset S_{a_1}$ by definition. By the same argument we get 
$\delta (2S_{a_1},Y) =\delta (2S_{a_1},H)$. If $c_{a_1}$ differs from $a$ and from $b$ in at least $3$ coordinates, then the
inductive assumption gives $\delta( 2S_{a_1},H)=0$ and hence $\delta (2S,Y)=0$. We conclude since $k>7 $ and $\#\pi_i(S)=2 $ for all $i $, so not all pairs of points of $ S$ may differ in only $ 3$ coordinates.

Thus we proved that for all $k\geq 7 $, then $\delta(S_{a_1},Y)=0 $, so by the semicontinuity theorem for cohomology, for all $k \geq 7$ we get $\delta (S,Y)=0 $ contradicting the assumption.
\end{enumerate}
\end{enumerate}\vspace{-0.5cm}
\end{proof}

\begin{lemma}\label{a4}
Let $Y =\PP^2\times \PP^2\times \PP^1$. Then each $S\in \TT(Y,3)$ is as in Example \ref{a4.0} for $k=3$ and
$n_1=n_2=2$.
\end{lemma}

\begin{proof}
 Set $\Uu:=\{ S\in S(Y,3) \,\vert\, Y \mbox{ is the minimal multiprojective space containing
S} \}$. So any $S\in \Uu $ is such that
$\#\pi _3(S)\ge 2$, $\pi _{1|S}$ and $\pi _{2|S}$ are injective and $\dim \langle \pi _1(S)\rangle =\dim \langle \pi
_2(S)\rangle =2$. The group $\mathrm{Aut}(\PP^2)\times \mathrm{Aut}(\PP^2)\times \mathrm{Aut}(\PP^1)$ acts on $\Uu$ with
exactly $2$ orbits:
\begin{enumerate}
\item $\#\pi _3(S) =3$;
\item $\#\pi _3(S) =2$.
\end{enumerate}
Call $O_1$ the first orbit and $O_2$ the second one. Obviously $h^1(\Ii _{(2S,Y)}(1,1,1)) = h^1(\Ii _{(2S',Y)}(1,1,1))$ for all
$S, S'$ in  the same orbit. Among the elements of $O_1$ there are the general subset of $Y$ with cardinality $3$. Since
$\sigma _3(\nu(Y)) =\PP^{17}$ 
(\cite{CGG05}) $h^1\left(\Ii _{(2S,Y)}(1,1,1)\right) =0$ for all $S\in O_1$. Note that the elements of
$O_2$ are exactly the sets $S$ described in Example \ref{a4.0} for $n_1=n_2=2$ and $k=3$.
\end{proof}

\begin{lemma}\label{lu3}
Let $Y =\PP^2\times \PP^2\times \PP^2$. The $3$-rd Terracini locus  $ \TT(Y,3)$ is empty. 
\end{lemma}

\begin{proof}
Let $ S\in S(Y,3)$ be such that $ Y$ is the minimal multiprojective space that contains $ S$, i.e. $\pi_{i\vert S} $ is injective and $\dim \langle \pi_i(S) \rangle =2$ for all $i$. By the action of $\Aut(\mathbb{P}^2)^3 $, we can reduce to work with a general set $ S\in S(Y,3)$. Since $\sigma_3(X) $ is not defective (cf. \cite[Example 4.1]{CGG2}) we know that $ \dim(\sigma_3(X))=20$, so $ h^0(\mathcal{I}_{(2S,Y)}(1,1,1))=6$. Hence, by the restriction exact sequence, $\delta(2S,Y)=0$.
\end{proof}

\begin{lemma}\label{b1}
Let $ Y=\mathbb{P}^2\times \PP^1\times \PP^1$. If $ S \in \TT(Y,3)$ then $ S$ is either as in Example \ref{ex1} or as in Example \ref{a4.0}.
\end{lemma}

\begin{proof}
Let $S\in S(Y,3) $ such that $ Y$ is the minimal multiprojective space containing $ S$, i.e. $ \pi_{1 \vert S}$ is injective, $ \dim \langle \pi_1(S) \rangle=\mathbb{P}^2$ and $\# \pi_i(S)\geq 2$ for all $i\in \{2,3\} $. We remark that $S$ is as in Example  \ref{a4.0} or as in Example \ref{ex1} if and only if there exists an index $ i\in \{2,3\}$ such that $\#\pi_i(S)=2 $. 

Assume by contradiction that $ S$ is neither as in Example \ref{a4.0} nor Example \ref{ex1}, i.e. assume that $ \#\pi_i(S)=3$ for $ i=2,3$. Since $\mathrm{Aut}(\PP^2)$ is transitive on the set of triples of linearly independent points of $\PP^2$ and $\mathrm{Aut}(\PP^1)$ is 3-transitive, $S$ is in the open orbit for the action of $\mathrm{Aut}(\PP^2)\times \mathrm{Aut}(\PP^1)\times \mathrm{Aut}(\PP^1)$ on $S(Y,3)$. So we can deal with a general set $S\in S(Y,3) $. 
By \cite[Theorem 4.5]{aop} we know that $\sigma _3(X)$ is non-defective, so $\sigma_3(X)=\PP^{11} $ and hence $ h^0\left(\mathcal{I}_{(2S,Y)}(1,1,1)\right)=0$, contradicting the assumption.
\end{proof}

\begin{lemma}\label{ju1}
Let $Y:=\mathbb{P}^{n_1}\times\cdots \times \mathbb{P}^{n_k} $, where $k\geq 5 $ and $n_i\in \{1,2 \} $ for all $i$'s. If $S\in \TT(Y,3) $ then $ S$ is either as in Example \ref{ex1} or as in Example \ref{a4.0}. 
In particular $\TT(Y,3)=\emptyset$, unless $n_i=1$ for at least $k-2$ indices $i$.
\end{lemma}

\begin{proof}
We proceed by induction on the integer $t:= \dim Y-k$.\\
The base case $t=0 $ corresponds to Lemma \ref{o3}. Assume $t>0$ and that the lemma is true for any multiprojective space $ Y$ of dimension at most $ k+t-1$. Since $t>0$, there exists at least an index $ i$ such that $n_i=2$, without loss of generality we may assume $i=1$. Fix
$S\in
\TT(Y,3)$. So we know that $\delta (2S,Y)>0$ and
$Y$ is the minimal multiprojective space containing $S$. Thus $\pi _{1|S}$ is injective and $\langle \pi_1(S)\rangle = \PP^2$.
Fix $o\in \PP^2\setminus \pi_1(S)$. Choose a system of homogeneous coordinates $\{x_0,x_1,x_2\}$ of $\PP^2$ such that $o =[1:0:0]$, the line $L:= \{x_0=0\}$ contains no point of $\pi_1(S)$ and $o$ is not contained in one of the $3$ lines spanned by $2$ of
the points of $\pi_1(S)$. Let $\ell _o:\PP^2\setminus \{o\}\to L$ denotes the linear projection from $o$, i.e. the rational map defined by $[a_0:a_1:a_2]\mapsto [0:a_1:a_2]$. 

Write $Y =\PP^2\times Y'$ with $Y' =\prod_{i>1}\PP^{n_i}$ and set
$H:= L\times Y'\in |\Oo_Y(\epsilon _1)|$. The morphism $\ell _o$ extends to a morphism 
\begin{align*}
f_o: Y& \longrightarrow H\\
(a,b)&\mapsto (\ell _o(a),b),
\end{align*}
which is defined for any $a\in \PP^2\setminus
\{o\}$ and any $b\in Y'$. We remark that $\#f_o(S) =3$ and that $H$ is the minimal multiprojective subspace of $Y$
containing $f_o(S)$.\\ For each $\lambda \in \KK\setminus \{0\}$ let $u_\lambda: \PP^2\to \PP^2$ denotes the automorphism of
$\PP^2$ defined by the formula $[a_0:a_1:a_2]\mapsto [\lambda a_0:a_1:a_2]$. Let $\KK'\subseteq \KK\setminus \{0\}$ be the
set of all $\lambda \in \KK\setminus \{0\}$ such that no line spanned by $2$ of the points of $u_\lambda (\pi _1(S))$
contains $o$. For each $\lambda \in \KK'$ we have $\#u_\lambda(\pi_1(S))=3$ and $u_\lambda(\pi_1(S))$ spans $\PP^2$. \\For each $\lambda
\in \KK'$ define
\begin{align*}
g_\lambda : Y& \longrightarrow Y\\
(a,b)&\mapsto (u_\lambda (a),b).
\end{align*}
Composing $f_o$ with the inclusion $j: H\subset Y$ we see that the rational map $j\circ f_o$ is a
limit for $\lambda$ going to $0$ of the family $\{g_\lambda\}_{\lambda \in \KK'}$ of automorphisms of $Y$. By the
semicontinuity theorem for cohomology $\delta (2(j\circ f_o(S)),Y)\ge \delta (2S,Y)>0$.
\begin{quote}
    \begin{claim}\label{claimju1}
$\delta (2g_0(S),H) = \delta (2(j\circ f_o(S)),Y)$.
\end{claim}
\begin{proof}
Since $\dim Y =\dim H+1$, part (a) of Lemma \ref{a1} gives  $\delta (2g_0(S),H) \le \delta
(2(j\circ f_o(S)),Y)
\le
\delta (2g_0(S),H)+h^1\left(\Ii _S(\hat{\epsilon}_1)\right) $. To conclude the proof of Claim \ref{claimju1} it is sufficient to prove that $h^1\left(\Ii
_S(\hat{\epsilon}_1)\right)=0$. Assume $h^1\left(\Ii _S(\hat{\epsilon}_1)\right)>0$. By \cite[Lemma 4.4]{BBS} either there are $u, v\in S$
such that $u\ne v$ and $\eta_1(u)=\eta _1(v)$ or there is $i\in \{2,\dots ,k\}$ such that $\#\pi _h(S) =1$ for all $h\in
\{2,\dots ,k\}\setminus \{i\}$. In the former case, i.e. if $\pi _i(u)=\pi_i(v)$ for all $i>1$, $S$ is as in Example \ref{ex1}.
In the second case we are either in Example \ref{a4.0} or in Example \ref{ex1}.
\end{proof}
\end{quote}

By Claim \ref{claimju1} and the inequality $h^0\left(\Oo_H(1,\dots ,1)\right) > 3\dim H$ (true because $k\ge 5$) $f_o(S)\in \TT(H,3)$. By the
inductive assumption
$f_o(S)$ is as in one of the Examples
\ref{a4.0} or
\ref{ex1} and in particular $n_h=1$ for at least $k-2$ of the last $(k-1)$ indices $h$, say for $h\in \{3,\dots ,k\}$. Moreover
there is $A\subset f_o(S)$ such that $\#A=2$ and $\#\pi_h(A)=1$ for all $h>2$. Since $f_0$ act as the identity on the last
$(k-1)$ components of any $p\in Y\setminus H$, we get that $S$ is described by the same Example which describes $f_o(S)$.
\end{proof}

\begin{lemma}\label{lu5}
Take $Y =\PP^{n_1}\times \PP^{n_2}\times \PP^{n_3}\times \PP^{n_4}$ with $n_i\in \{1,2\}$ for all $i$
and $n_1+n_2+n_3+n_4\ge 5$. If $S\in \TT(Y,3)$, then $S$ is either as in Example \ref{a4.0} or as in Example \ref{ex1}.
\end{lemma}

\begin{proof}
We will show the result by induction on the integer $ t=n_1+\cdots +n_5-5\geq0$. First assume $ t=0$, i.e. $n_1+n_2+n_3+n_4=5$.
 With no loss of generality we may assume $Y =\PP^2\times \PP^1\times \PP^1\times \PP^1$.
Since $Y$ is the minimal multiprojective space containing $S$ and $n_1=2$, $\pi _{1|S}$ is injective. Assume for the moment $\pi _{i|S}$ injective for $i=2,3,4$.
Since $\mathrm{Aut}(\PP^1)$ is $3$-transitive, $S$ is in the same orbit for the action of $\mathrm{Aut}(\PP^2)\times \mathrm{Aut}(\PP^1)^3$ of $3$ general points of $Y$. 
We know that
 $\dim \sigma _3(\nu(Y)) = 17$ (\cite[Theorem 4.5]{aop}),  so $\delta(2S,Y)=0 $ contradicting the assumption. Thus we may assume $\#\pi _i(S)=2$ for some $i\in 2,3,4$. With no loss of generality we may assume that at least
$\#\pi _3(S)=2$. Since $\pi _{1|S}$ is injective, $\eta _{4|S}$ is injective. The set $ \eta_4(S)$ is as in case \ref{a4.0itemv} of Example \ref{a4.0} whether $ \#\pi_i(S)=2$ just for the index $ i=3$ or not.
Using $\eta _2$ and $\eta _3$ instead of $\eta _4$ we see the existence of at least two indices $h\in \{2,3,4\}$ such that
$\#\pi _h(S)=2$. With no loss of generality we may assume $\#\pi _3(S)=\#\pi _4(S)=2$, i.e. neither $\pi _{3|S}$ nor
$\pi_{4|S}$ are injective.  If there is $S'\subset S$ such that $\#S'=2$ and $\#\pi _3(S')=\#\pi _4(S')=1$, then we are in
Example \ref{a4.0} or Example \ref{ex1}. The non-existence of such $S'$ shows that we may name $S =\{a,b,c\}$ so that $\pi
_4(a)=\pi _4(b)$, $\pi_3(a) =\pi _3({c})$.  We distinguish two cases:
\begin{enumerate}[label=(\roman*)]

\item\label{lu51} $\#\pi _2(S) =2$;
\item\label{lu52}  $\#\pi _2(S) =3$.
\end{enumerate}
Write $a =[a_1,a_2,a_3,a_4]$, $b=[b_1,b_2,b_3,b_4]$ and $c=[c_1,c_2,c_3,c_4]$.
Since $\mathrm{Aut}(\PP^2)$ is transitive on the set of all triples of linearly independent points, we may assume
$a_1 =[1:0:0]$, $b_1=[0:1:0]$ and $c_1=[0:0:1]$. Since
$\mathrm{Aut}(\PP^1)$ is
$3$-transitive we may assume $a_2=a_3=a_4=\alpha$, $b_3=\beta$, $b_4=\alpha$, $c_3=\alpha$ and $c_4=\beta$, for some $\alpha\neq \beta \in \mathbb{P}^1 $.  Moreover, in case \ref{lu51} we may assume
$b_2=c_2=\beta$, while in case \ref{lu52} we may assume  $b_2=\beta$ and $c_2=\gamma$, for some $ \gamma \in \mathbb{P}^1 $ with $\gamma \neq \alpha,\beta$. For both cases, by direct computation one gets  $h^0\left(\mathcal{I}_{(2S,Y)}(1,1,1,1)\right)=17$, so $\delta(2S,Y)=0 $ contradicting the assumption.

Now assume $ t>0$, i.e. $n_1+n_2+n_3+n_4\ge 6$. As in the proof of Lemma \ref{ju1}, we will use a linear projection from a general point of a
$2$-dimensional factor of $Y$ to conclude by induction on the integer $n_1+n_2+n_3+n_4$.

\end{proof}

\begin{theorem}\label{i33}

Let $Y$ be the minimal multiprojective space of $k\geq 1$ factors containing a set $S$ of 3 points. Then $\TT(Y,3) $ is empty if and only if either $Y= (\mathbb{P}^2)^i$, $i=1,2,3$ or $Y= \PP^{n_1}\times \PP^{n_2} $, $n_1,n_2 \in \{1,2\}$. Moreover the non-empty $S \in \TT(Y,3) $ can only be either as in Example \ref{ex1} or as in Example \ref{a4.0} or $Y=(\PP^1)^4$, in this last case $ S(Y,3) \subset \TT(Y,3)$.

\end{theorem}

\begin{proof}

Let $S\in \TT(Y,3) $ such that $ Y$ is the minimal multiprojective space containing $ S$, so $Y=\PP^{n_1}\times \cdots \times \PP^{n_k} $ where all $n_i\in \{1,2 \} $.

If $k=1$ we always have $h^0(\Ii _{2S}(1)) =0$, thus the case of $Y=\mathbb{P}^2$ is clear.

Assume $k=2$. In this case $Y =\PP^{n_1}\times \PP^{n_2}$ with $1\le n_1\le 2$ and $1\le n_1\le 2$. If $n_1=n_2=1$, then obviously $h^0(\Ii _{2S}(1,1))=0$.
If $n_i=2$, then $\pi _{i|S}$ is injective
and $\pi _i(S)$ is linearly independent.

Thus if $n_1=n_2=2$, then $S$ is in open orbit for the action of $\mathrm{Aut}(\PP^2)\times \mathrm{Aut}(\PP^2)$  of $S(Y,3)$. Since a general $3\times 3$ matrix has rank $3$ we get $\sigma _3(\nu (Y)) = \PP^8$. Hence $h^0(\Ii_{(2S,Y)}(1,1)) =0$, contradicting the assumption $S\in \TT(Y,3)$. 

Now assume $n_i=1$ for exactly one $i$, say for $i=1$. Since $Y$ is the minimal multiprojective space containing containing $S$, $\#\pi_i(S) \ge 2$. Fix $S'\subset S$ such that $\#S' =\#\pi _1(S') =2$. $S'$ is in the open orbit for the action of $\mathrm{Aut}(\PP^1)\times \mathrm{Aut}(\PP^2)$ on $S(Y,2)$. Since a general $2\times 3$ matrix has rank $2$,
$\sigma _2(\nu(Y)) =\PP^5$. Thus $h^0(\Ii _{(2S',Y)}(1,1)) =0$. Hence $h^0(\Ii_{(2S,Y)}(1,1)) =0$, contradicting the assumption $S\in \TT(Y,3)$. 
This concludes the case of two factors.

The case of $k=3 $ is completely covered by Lemmas \ref{lu2}, \ref{lu3}, \ref{b1} and \ref{a4}.

In the case of $k=4$ there is
the defective 3-rd secant variety of the Segre embedding of $Y=(\PP^1)^4 $ (cf. Remark \ref{lu1}).

For any other couple $(S,Y) $ where $Y\ncong (\PP^1)^4 $, Lemma  \ref{lu5} shows that $ S$ must be either as in Example \ref{ex1} or as in Example \ref{a4.0}.

If $k\ge 5$ it is sufficient to use Lemma \ref{ju1}.
\end{proof}

\section{Computing the maximal $ r$-th Terracini defect}
\phantom{a}

 Fix any multiprojective space $ Y $ of dimension $ n>0$. 
 For any $p\in Y$ the very ampleness of $\Oo_Y(1,\dots ,1)$ implies $h^1\left(\Ii_{(2p,Y)}(1,\dots
,1)\right)=0$.
For any integer $r\ge 2$ there are many $S\in S(Y,r)$ with $\delta(2S,Y) >0$. We will show  that the maximal value of all such integers $\delta(2S,Y) >0$ for some multiprojective space $ Y$ of dimension $n$, is obtained when $Y =\PP^n$ (cf. Proposition \ref{g1}). But of course
$h^0\left(\PP^n,\Ii _{(2S,\PP^n)}(1)\right) =0$ for any finite set $S\subset \PP^n$ with $S\ne \emptyset$. 

\begin{definition}
For any integer $n>0$, denote by $\Uu(n)$ the set of all isomorphism classes of multiprojective spaces $Y$ such that $\dim
Y=n$.

For any integer $r\ge 2$, $n\ge
2$ define 
$$\Ee (n,r):=\{ (Y,S)\in \Uu(n) \times S(Y,r) \; \vert \; S \in \TT_1(Y,r)\},$$
$$\EE (n,r):=\{ (Y,S)\in \Uu(n) \times S(Y,r) \; \vert \; S \in \TT(Y,r)\}.$$
\end{definition}

 The set of all
$(n,r)$ such that
$\Ee(n,r)\ne
\emptyset$ is easily computed in Lemma \ref{g2} 
and we will show that $\Ee(n,r)\ne \emptyset$ if and only if $n\ge 3$ and $r\ge 2$.

\begin{notation}
Fix integers $n,r>0 $. Denote by 
$$\delta_1(n,r):=\max \{ \delta(2S,Y)\; \vert \; (Y,S)\in \Ee(n,r)  \}. $$ 
\end{notation}
We remark that given any $S\in S(Y,r) $ such that $h^0\left(\Ii _{(2S,Y)}(1,\dots ,1)\right) >0$, asking whether $ S\in \TT_1(Y,r)$ is equivalent to request that $ \delta(2S,Y)>0$. Similarly, if $S\in S(Y,r) $ is such that $ \delta(2S,Y)>0$, then to show that $S \in \TT_1 (Y,r)$ it suffices to prove $h^0\left(\Ii_{(2S,Y)}(1,\dots,1)\right)>0 $. In Proposition \ref{g3} we will show that 
$$\delta_1(n,r) =(r-1)(n+1)-1.$$ 
If we also prescribe that $(Y,S)\in \EE(n,x)$, i.e. if we assume that
$Y$ is the minimal multiprojective space containing $S$, then we get the definition of the integer $\delta (n,x)$.

\begin{proposition}
\label{g1}
Fix integers $n>0$ and $r\ge 2$. Fix $Y\in \Uu(n)$ and $S\in S(Y,r)$. Then
$$h^1\left(\Ii_{(2S,Y)}(1,\dots,1)\right)\leq (r-1)(n+1) .$$
The equality holds if and only if $ Y=\PP^{n}$.
\end{proposition}

\begin{proof}

Fix $Y\in \Uu (n)$, say $Y =\PP^{n_1}\times \cdots \times  \PP^{n_k}$ with $n_i>0$ for all $i$'s and $n_1+\cdots +n_k =n$ and assume $k\geq 2 $, i.e. assume $Y\ncong \PP^n$.
Fix $S\in S(Y,r)$ and take $o\in S$. Since $\Oo _Y(1,\dots ,1)$ is very ample, we have $h^1\left(\Ii _{(2o,Y)}(1,\dots ,1)\right) =0$. Thus 
by \eqref{eqm2} of Remark \ref{m1} we get 
 $h^1\left(\Ii _{(2S,Y)}(1,\dots ,1)) \le \deg (2(S\setminus \{o\}),Y)\right) =(r-1)(n+1)$, concluding the proof of the inequality. 
 
The ``\,if\,'' part of the equality is clear, so we just need to prove the ``\,only if\,'' part. We will use induction on the integer $n$ starting with the case $n=2$.

Let $ n=2$ and assume by contradiction that $ Y\ncong \PP^2$, so $Y=\PP^1\times \PP^1 $.
Thus $h^0\left(\Oo _Y(1,1)\right) =4$. Since each tangent plane of $\nu (\PP^1\times \PP^1)$ is tangent to a unique point of the smooth quadric $\nu (\PP^1\times \PP^1)$  and $ r\geq2$, we have $h^0\left(\Ii _{(2S,Y)}(1,1)\right) =0$ and hence  $h^1\left(\Ii _{(2S,Y)}(1,1)\right) =3(r-1)-1\neq 3(r-1)$. Now assume $n>2$. 
We distinguish two different cases depending on whether $ r=2$ or not.

\begin{enumerate}[leftmargin=+.2in, label=(\alph*)]
\item\label{g1induzionea} Assume $r=2$ and write $S =\{u,v\}$. Assume by contradiction that $Y\ncong \PP^n $. We remark that $h^1\left(\Ii _{(2S,Y)}(1,\dots ,1)\right) = n+1$. This implies that the Zariski tangent spaces $T_{\nu(u)}\nu(Y)$ and $T_{\nu(v)}(Y)$ of $\nu(Y)$ at $\nu (u)$ and $\nu(v)$ are the same.  Since $\nu (v)\in T_{\nu(u)}\nu (Y)$, the line $L:= \langle \{\nu (v),\nu (u)\}\rangle$  contains two points of $T_{\nu (u)}\nu(Y)$ and hence it is contained in $T_{\nu(u)}\nu (Y)$. Since $\nu(u)\in L$, $L$ is tangent to $\nu (Y)$ at $\nu(u)$. Hence $L\cap \nu(Y)$ contains a zero-dimensional scheme of degree strictly grater than $2$.
Since $\nu (Y)$ is scheme-theoretically cut out by quadrics, we get $L\subset \nu (Y)$, i.e. there is $D\subset Y$, such that $\nu (D)=L$, $D\cong \PP^1$ and $\#\pi _i(D)=1$ for $k-1$ indices $i$.  Let $i \in \{ 1,\dots,k\} $ be the index such that $ \#\pi_i(S)\neq 1$.
Since $\# \pi_j(S)=1 $ for all $j\neq i $ and $ S\subset D$, by case \ref{a1unfattore} of Lemma \ref{a1}, we know that $\delta(2S,Y)=\delta(2S,\PP^{n_i}) $, where $n_i<n $. By the inductive assumption we get $\delta(2S,\PP^{n_i})=n_i+1<n+1 $ which is absurd since by assumption $\delta(2S,Y)=n+1 $.
 
\item Assume $r>2$. Write $S = A\cup B$ with $\#A=2$ and $\#B = r-2$. By part \ref{g1induzionea} we have $h^1\left(\Ii _{(2A,Y)}(1,\dots ,1)\right) \le n$. Thus by \ref{eqm2} of Remark \ref{m1} we get $h^1\left(\Ii _{(2S,Y)}(1,\dots ,1)\right)
\le h^1\left(\Ii _{(2A,Y)}(1,\dots ,1)\right)+\deg (2B,Y) \le n+(r-2)(n+1)$, which is absurd since by assumption $h^1\left(\Ii_{(2S,Y)}(1,\dots,1)\right)=(r-1)(n+1) $.
\end{enumerate}
\vspace{-0.5cm}
\end{proof}

\begin{example}\label{g0}
Let $ n\geq 3$, fix an integer $ 1\leq \mu \leq n-1$ and let $ r\geq \mu +1$. Let $L\subset \PP^{n-1} $ be a $ \mu$-dimensional linear subspace and let $ Y:=\PP^{n-1}\times \PP^1$. Fix $ o\in \PP^1$ and a finite set $ S\subset L\times \{ o\}$ with  $\#S=r$ and such that $ \langle \pi_1(S) \rangle=L$. The aim of this example is to show that
$\delta (2S,Y) = (r-1)(n+1)-\mu$. 

Take $H:= \pi _2^{-1}(o)\in |\Oo _Y(\epsilon _2)|$. Note that $S\subset H$. Thus the residual exact sequence of $ (2S,Y)$ with respect to $ H$ is
\begin{equation}\label{u+1}
0 \to \Ii_S(1,0) \to \Ii_{(2S,Y)}(1,1)\to \Ii_{(2S,H),H}(1,1)\to 0.
\end{equation}
We remark that $S\neq \emptyset $ and in particular $\# S\geq 2 $. Moreover $ H\cong \PP^{n-1}$, so $ h^0\left(H,\Ii_{(2S,H)}(1,1)\right)=0$. Since by assumption $\langle \pi_1(S)\rangle=L$, where $\dim L =\mu$, we get  $h^0\left(\Ii_S(1,0)\right) = n-1-\mu$. So by \eqref{u+1} we get $h^0\left(\Ii_{(2S,Y)}(1,1)\right) = n-1-\mu$. Thus $\delta(2S,Y) = r(n+1) -2n+n-\mu-1 = (r-1)(n+1)-\mu$.

In particular for $\mu =1$, i.e. if $L$ is a line, we obtain $\delta (2S,Y) =(r-1)(n+1)-1$. Since $h^0(\Oo_Y(1,1)) =2n$
and $\deg (2S,Y) =r(n+1)$, when $\mu =1$ we get $h^0(\Ii _{(2S,Y)}(1,1)) =2n -r(n+1) +(r-1)(n+1) -1 =n-2>0$. Thus if $\mu=1$, $S\in
\TT_1(Y,r)$ and in particular $\delta _1(S,Y) = (r-1)(n+1)-1$.\\
Obviously also $\PP^1\times \PP^{n-1}$ gives an example, taking an $ L$
in the second factor of $Y$.
\end{example}

\begin{lemma}\label{g2}
Fix integers $n\ge 2$ and $r\ge 2$. $\Ee (n,r)\ne \emptyset$ if and only if $n\ge 3$.
\end{lemma}

\begin{proof}
For $n=2$ we remark that  $\Uu (2)= \{[\PP^2],[\PP^1\times \PP^1]\}$. For both cases, by Proposition \ref{g1} we get $ h^0\left(\Ii_{(2S,Y)}(1,\dots,1)\right)=0$. Viceversa, if $n\geq 3 $ we may take $Y =\PP^{n-1}\times \PP^1$ and $S$ as in Example \ref{g0}.
\end{proof}

\begin{remark}\label{remg3}Let $n>0$ and $r\geq 2$. By Proposition \ref{g1}, for all $Y\in \Uu(n) $ and $ S\in S(Y,r)$ the maximum value of $ h^1\left(\Ii_{(2S,Y)}(1,\dots,1)\right)$ is achieved when $ Y=\PP^{n}$. Clearly if $Y=\PP^n $, $ h^0\left(\Ii_{(2S,Y)}(1,\dots,1)\right)=0$.
Thus the couple $(Y,S)\in \Uu(n)\times S(Y,r) $ evincing $\delta_1(n,r) $ is such that $ Y$ is a multiprojective space with $ k\geq 2$ factors.
\end{remark}

\begin{theorem}\label{g3}
Fix integers $n\ge 3$ and $r\ge 2$. Then $\delta _1(n,r) = (r-1)(n+1)-1$ and any $(Y,S)$ evincing $\delta _1(n,r)$ is as in
Example \ref{g0} with $\mu =1$.
\end{theorem}

\begin{proof}

By Remark \ref{remg3} we may work with multiprojective spaces $Y $'s of $k\geq 2 $ factors. So, by Proposition \ref{g1}, for all $(Y,S) $ 
$$ \delta(2S,Y)\leq (r-1)(n+1)-1.$$ 
The case $\mu =1$ of Example \ref{g0} gives the inequality $\delta _1(n,r) \ge (r-1)(n+1)-1$. Thus it remains to prove that this is the only case.

Fix $(Y,S)$ evincing $\delta
_1(n,r)$. Thus
$Y =\PP^{n_1}\times \cdots \times \PP^{n_k}$ 
where all $ n_i>0$ and are such that $n_1+\cdots +n_k=n$. The finite set $S\in S(Y,r)$, 
is such that $h^0\left(\Ii
_{(2S,Y)}(1,\dots ,1)\right) >0$ and $h^1\left(\Ii _{(2S,Y)}(1,1)\right) \ge (r-1)(n+1)-1$.

We will show the result by induction on $ n\geq 3$.

If $n=3 $ then $\Uu(3)=\{[ \PP^3],[\PP^2\times \PP^1], [(\PP^1)^3]\} $.

Clearly the case $ Y=\PP^3$ is excluded by Remark \ref{remg3}. If $Y=\PP^2\times \PP^1 $, it suffices to show that for any other $ r$-uple of points $ \hat{S}\in S(Y,r)$ that is not as in Example \ref{g0}, we get $\delta(2\hat{S},Y)< 4(r-1)-1  $. If $ r=2$ this is true since $\delta(2S,Y)= 2$ unless $ S\in S(Y,2)$ is as in Example \ref{g0}. If $ r\geq 3$ then $ h^0\left(\Ii_{(2S,Y)}(1,1)\right)=0$ for all $ S\in S(Y,r)$.\\
Let $ Y=(\PP^1)^3$. By Proposition \ref{d1} we exclude the case $r=2 $ since either $\delta(2S,Y) $ or $h^0\left(\Ii_{(2S,Y)}(1,1,1)\right) $ is zero. If $ r=3$, Lemma \ref{lu2} gives the only cases for which $ S\in \TT(Y,3)$ and for such cases we already proved that $\delta(2S,Y)=5< \delta_1(3,3) $ and $ h^0\left(\Ii_{(2S,Y)}(1,1,1)\right)=1$. Thus for $r\geq 4 $ we get $h^0\left(\Ii_{(2S,Y)}(1,1,1)\right)=0 $ for all $ S\in S(Y,r)$ that are not as in Example \ref{g0}.

Assume that the proposition is true for all $n'<n$. 
 We will prove the inductive step by induction on $r\geq 2 $. Case \ref{g3base} will be the base case and in case \ref{g3induzione} we will show the inductive step.

\begin{enumerate}[leftmargin=+.2in, label=(\alph*)]
\item\label{g3base} 
 Assume $r= 2$ and let $L:=\langle \nu(S)\rangle $.

 First assume that $ Y$ has $k=2$ factors, i.e. $ Y=\PP^{n_1}\times \PP^{n_2}$. With no loss of generality we may assume $n_1\ge n_2$. To conclude 
this case it is sufficient to prove that $n_2 =1$ and $\#\pi _2(S)=1$ and we will do it by contradiction. 

First assume $n_2\ge 2$. Since $h^0(\Oo _Y(0,1)) =n_2+1 > 2 $, there is $M\in |\Ii_S(0,1)|$. Thus $S \subset M$.  If $ (S,M)$ is as in Example \ref{g0} there is nothing to prove, otherwise by the inductive step we get $h^1(M,\Ii _{(2S,M)}(1,1)) \le n-2$.
Since $\dim Y=\dim M +1$, part \ref{a1disuguaglianze} of Lemma \ref{a1} gives  $h^1(\Ii _{(2S,Y)}(1,1))\le n-2+1 <n$  which is absurd since we took $ (Y,S)$ evincing $ \delta_1(2,n)=n$. 

Assume now that $ \#\pi_2(S)=2$. Again if $\#\pi_1(S)=1 $ then $S$ is as in  Example \ref{g0}, so assume also $ \#\pi_2(S)=2$. Thus the minimal multiprojective space containing $ S$ is $Y=\PP^2\times \PP^2 $. So $ S$ is in the open orbit for the action of  $Aut(\PP^2)^2$ on $ S(Y,3)$. Hence $h^0\left(\Ii_{(2S,Y)}(1,1)\right)=0$ and consequently, since $ \deg(2S,Y)=15$ and $ h^0\left(\Oo_Y(1,1)\right)=9$, we get $\delta(2S,Y) =6<\delta_1(4,3)$.

Assume now $ Y$ has $k>2$ factors. 
By Lemma \ref{a1} and the equality $\delta _1(n',2) = (r-1)(n'+1)-1$ for all $n'<n$, $Y$ is the minimal multiprojective space containing $S$. Thus $Y = (\PP^1)^k$.
Fix $H\in |\Oo_Y(\epsilon _k)|$ containing at  least on
point of $S$. Since $S\nsubseteq H$, $\#(S\cap H) =\#(S\setminus S\cap H) =1$. Denote by $S:=\{a,b \} $ and by relabeling if necessary, assume $ S\cap H=\{ a\}$ and $S\setminus S\cap H=\{ b\} $.

 Consider the residual exact sequence of $H$:
\begin{align}\label{eqg1}
&0 \to \Ii _{(2b,Y)\cup (a,Y)}(\hat{\epsilon}_k) \to \Ii_{(2S,Y)}(1,\dots ,1)\to \Ii _{(2a,H),H}(1,\dots ,1)\to 0.
\end{align}
Since $\#(S\cap H)=1$ and $\Oo_H(1,\dots ,1)$ is very ample, $\delta(2a,H)=0$. 
Since $\#(S\setminus S\cap H) =1$, $\Oo_{Y_k}(1,\dots ,1)$ is very ample and $\dim Y-\dim Y_k=1$, we
have
$h^1\left(\Ii _{(2b,Y)}(\hat{\epsilon}_k)\right)=0$. Since $\#(S\cap H)=1$, $h^1\left(H, \Ii _{(2b,Y)\cup (a,Y)} (\hat{\epsilon}_k)\right)\le 1$. Thus \eqref{eqg1} gives $h^1\left(\Ii _{(2S,Y)}(1,\dots ,1)\right) \le 1<n$, a contradiction.

\item\label{g3induzione} Assume now $r\ge 3$. Fix any $A\subset S$ such that $\#A=r-1$. Since $\delta (2S,Y)\le \delta
(2A,Y) +n+1$ (cf. Remark \ref{m1}), the inductive assumption gives the pair $(Y,A)$ is as in Example \ref{g0}. Thus either $Y\cong
\PP^{n-1}\times
\PP^1$ or
$Y\cong \PP^1\times \PP^{n-1}$. With no loss of generality we may assume $\PP^{n-1}\times \PP^1$. The inductive
assumption gives the existence of a line $L_A\subset \PP^{n-1}$ and a point $o_A\in \PP^1$ such that $A\subset L\times
\{o_A\}$. Since $r\ge 3$ there is $B\subset S$ with $\#B=r-1$, $B\cap A\ne \emptyset$ and $B\ne A$. We get $\{o_A\} =
\pi_2(A)=\pi _2(B) =\{o_B\}$. Thus $\#\pi _2(S)=1$. 

To conclude the proof it is sufficient to show that $\pi_1(S) $ spans a line 
and we will do it by induction on $ r\geq 3$. Take for the moment $r=3$, assume that
$\langle \pi_1(S)\rangle$ is a plane and set $M:= \langle \pi _1(S)\rangle \times \PP^1$. By part \ref{a1disuguaglianze} of
Lemma \ref{a1} and the assumption $\delta (2S,Y) =2(\dim Y+1)-1$, we have $\delta (2S,M) \ge 2(\dim M+1)-1$. Moreover  $\delta
(2S,M) = 2(\dim M+1)-1 =7$, because $M$ is not a projective space. However, by direct computation, one gets $\delta (2S,M) =
3(\dim M +1)-6 =6$.

Let $r\ge 4$. Take any $2$ distinct subsets $A$, $B$ of $r$ with $\#A=\#B =r-1$. Since $\#A\cap B =r-2\ge 2$,
the lines $L_A$ and $L_B$ have at least $2$ common points. Thus $L_A=L_B$. Hence $\pi _1(S)$ spans a line.
\end{enumerate}
\vspace{-0,5cm}
\end{proof}

Example \ref{g0} gives the following result, the last equality being true by Theorem \ref{g3}.

\begin{theorem}\label{boo1}
Fix integers $n> \mu \ge 2$ and $r\ge \mu+1$. Then there is $(Y,S) \in \Ee (n,r)$ such that $\delta (2S,Y) = (r-1)(n+1)-\mu =\delta _1(n,r)-\mu+1$.
\end{theorem}

\begin{remark}\label{g4.0}
If $S\in \Ss(Y,2) $ is such that $ Y$ is the minimal multiprojective space containing $ S$, then $ Y=(\PP^1)^k$, for some $ k\geq 1$. Proposition \ref{d1} gives $\EE(n,2)=\emptyset$. 
Theorem \ref{i33} gives that $\EE(n,3)$ is the set of all $\PP^{n_1}\times
\PP^{n_2}\times (\PP^1)^{n-n_1-n_2}$ with $1 \le n_2\le n_1 \le 2$ and $n>n_1+n_2$. 
\end{remark}

\begin{example}\label{kk1}
Let $ r,n\geq 3$ and let $Y:=(\PP^1)^n $. Take $A\subset \PP^1$ such that $\#A =r-1$ and define $S:=\{p_1,\dots,p_r\} \subset Y $ where
\begin{align*}
   & p_i=(a_i,e_2,\dots,e_n) \mbox{ for } i=1,\dots,r-1 \mbox{ with all } a_i\in A \mbox{ and all } e_j\in \PP^1\\
  &  p_r=(o_1,\dots,o_r), \mbox{ with } o_1\in \PP^1\setminus A,\; o_k\in \PP^1 \; : \;  o_k\neq e_k \mbox{ for all } k=2,\dots,n 
\end{align*}

Set $S'=S\setminus \{ p_r\}$ and let $Y':=\PP^1\times \{ e_2\} \times \cdots \times \{e_n \}$. Note that $Y'\cong\PP^1$ is the minimal multiprojective subspace containing $S'$ and that $Y$ is the minimal multiprojective subspace containing
$S$. From \eqref{eqm2} of Remark \ref{m1} we know that $\delta (2S,Y)\ge \delta(2S',Y)\ge \delta (2S',Y')=\delta (2A,\PP^1) =2(r-2) >0$. Take $H:= \pi_n^{-1}(e_n)\in |\Oo
_Y(\epsilon _n)|$. Since $S'\subset H$, the residual exact sequence of $2S'$ with respect to $H$ gives

\begin{equation}\label{++eqk1}
0 \to \Ii_{S'}(1,\dots,1,0) \to \Ii_{(2S',Y)}(1,\dots ,1)\to \Ii _{(2S',H),H}(1,\dots ,1)\to 0
\end{equation} 

Thus \eqref{++eqk1} gives $h^0\left(\Ii_{(2S',Y)}(1,\dots ,1)\right) \ge h^0\left(\Ii_{S'}(1,\dots ,1)\right)$. Since $\nu(S')$ spans
a line, $h^0\left(\Ii_{S'}(1,\dots ,1)\right)=2^n -2$. Since $n\ge 3$, $h^0\left(\Ii_{(2S',Y)}(1,\dots ,1)\right) \ge n+2$. Thus $\delta (2S,Y)>0$ and $ h^0\left(\Ii_{(2S,Y)}(1,\dots,1)\right)>0$. 
\end{example}

\begin{proposition}\label{kk2}
$\EE(n,r)\ne \emptyset$ if and only if $n\ge 3$ and $r\ge 3$.
\end{proposition}

\begin{proof}
If $n\geq3 $ and $ r\geq3$, Example \ref{kk1} shows that $ \EE(n,r)\neq \emptyset$. The other implication follows from Lemma \ref{g2} since $ \EE(n,r)\subseteq \Ee(n,r)$ and $\TT(Y,2)=\emptyset $ (cf. Proposition \ref{d1}).
\end{proof}

\end{document}